\documentclass[12pt]{amsart}
\usepackage{a4wide,amsfonts,graphicx,
amssymb,stmaryrd,amscd,amsmath,latexsym,amsbsy,xypic}
\usepackage{marginnote}
\theoremstyle{plain}

\usepackage{etex}
\usepackage{amsfonts}
\usepackage{amsmath}
\usepackage{setspace}
\usepackage{mathrsfs}
\usepackage{caption}
\usepackage{subcaption}
\usepackage{stmaryrd}
\usepackage[english]{babel}
\usepackage{enumitem}
\usepackage{amssymb}
\usepackage{amsthm}
\usepackage[a4paper]{geometry} 
\usepackage{bm}

\newtheorem*{prop}{Proposition}
\newtheorem*{thm}{Theorem}
\newtheorem*{lem}{Lemma}
\newtheorem*{cor}{Corollary}
\newtheorem*{defi}{Definition}
\theoremstyle{remark}

\title{Quantum superintegrable spin systems on graph connections}
\author{Nicolai Reshetikhin}
\address{N.R.: YMSC, Tsinghua University, Beijing, China \& BIMSA, Beijing, China \& 
Department of Mathematics, University of California, Berkeley,
CA 94720, USA.}
\email{reshetik@math.berkeley.edu}

\author{Jasper Stokman}
\address{J.S.: KdV Institute for Mathematics, University of Amsterdam,
Science Park 105-107, 1098 XG Amsterdam, The Netherlands.}
\email{J.V.Stokman@uva.nl }
\begin{document}
\maketitle
\begin{abstract}
In this paper we construct certain quantum spin systems on moduli spaces of $G$-connections on a connected oriented finite graph, with $G$ a simply connected compact Lie group. We construct joint eigenfunctions of the commuting quantum Hamiltonians in terms of local invariant tensors. We determine sufficient conditions ensuring superintegrability of the quantum spin system
using irreducibility criteria for Harish-Chandra modules due to Harish-Chandra and Lepowsky \& McCollum.

The resulting class of quantum superintegrable spin systems includes the quantum periodic and open spin Calogero-Moser spin chains as special cases.
In the periodic case the description of the joint eigenfunctions 
in terms of local invariant tensors are multipoint generalised trace functions, in the open case 
 multipoint spherical functions on compact symmetric spaces.
 \end{abstract}
 \vspace{.5cm}
\begin{center}
{\it{Dedicated to the memory of Gerrit van Dijk}}
\end{center}
\vspace{.5cm}

\section{Introduction}
\subsection{}\label{1-1} 

Let $\Gamma$ be a connected oriented finite graph with vertex set $V$, edge set $E$, and source and target maps $s,t: E\rightarrow V$. Let $G$ be a connected compact Lie group. The product group $G^E=\{\bm{g}=(g_e)_{e\in E}\,\, | \,\, g_e\in G\}$ of 
 {\it graph $G$-connections} (or lattice gauge fields) on $\Gamma$ consists of colorings $\bm{g}$ of the edges of $\Gamma$ by group elements $g_e\in G$ ($e\in E$). We view $G^E$ both as compact Lie group and as algebraic group (via Tannaka duality). 
 
 The group 
$G^V=\{\bm{k}=(k_v)_{v\in V}\,\, | \,\, k_v\in G\}$ of lattice gauge transformations acts on $G^E$ by\footnote{With this convention of the action, $g_e\in G$ describes the holonomy along $e$ in the reverse direction.} 
\begin{equation}\label{actionintro}
(\bm{k}\cdot \bm{g})_e=k_{s(e)} g_ek_{t(e)}^{-1}.
\end{equation}
The resulting space $G^E/G^V$ of $G^V$-orbits in $G^E$ is the moduli space of graph $G$-connections on $\Gamma$ introduced by Fock and Rosly \cite{FR} to describe moduli spaces of flat connections on surfaces, see also \cite{AMR, AR}. See \cite{AGS,BR1,BR2,RSz} and references therein for the associated quantization problem.
In this paper we construct quantum systems with Hamiltonians being differential operators on the moduli space $G^E/\mathbf{K}$ of graph $G$-connections modulo  gauge groups $\mathbf{K}$ of the form $\mathbf{K}=\prod_{v\in V}K_v\subseteq G^V$, with $K_v$ arbitrary subgroups of $G$.

\subsection{}\label{1-2}
Let $\mathcal{D}(G^E)$ be the algebra of algebraic differential operators on $G^E$. The contragredient $\mathbf{K}$-action on the space of algebraic functions on $G^E$ 
induces a $\mathbf{K}$-action on $\mathcal{D}(G^E)$ by algebra automorphisms. We denote by $\mathcal{D}(G^E)^{\mathbf{K}}\subseteq\mathcal{D}(G^E)$ the subalgebra
of $\mathbf{K}$-invariant differential operators. In this paper we consider its subalgebras
\begin{equation}\label{inclusionalgintro}
\mathcal{D}_{\textup{biinv}}(G^E)\subseteq\mathcal{D}_{\textup{inv}}(G^E)^{\mathbf{K}}\subseteq\mathcal{D}(G^E)^{\mathbf{K}},
\end{equation}
with $\mathcal{D}_{\textup{inv}}(G^E)\subseteq\mathcal{D}(G^E)$ the subalgebra generated by the left and right $G^E$-invariant differential operators 
and $\mathcal{D}_{\textup{biinv}}(G^E)\subseteq\mathcal{D}(G^E)$ the subalgebra of $G^E$-biinvariant differential operators.

\subsection{}\label{1-3}
Let $\sigma: \mathbf{K}\rightarrow\textup{GL}(S)$ be a finite dimensional $\mathbf{K}$-representation.  
Functions in the space $\mathcal{H}=\mathcal{H}_{\Gamma,G,S}$ of algebraic sections of the associated vector bundle over $G^E/\mathbf{K}$ 
are called {\it spin graph functions}. They are algebraic functions $f: G^E\rightarrow S$ satisfying
\[
f(\bm{k}\cdot\bm{g})=\sigma(\bm{k})f(\bm{g}) 
\]
for $\bm{k}\in\mathbf{K}$ and $\bm{g}\in G^E$. Here spin refers to the interpretation of $S$ as spin space for the associated quantum spin system, see \S\ref{1-4} and  \S\ref{1-9}.

\subsection{}\label{1-4} 
The algebra $\mathcal{D}(G^E)^{\mathbf{K}}$ acts as scalar valued $\mathbf{K}$-invariant differential operators on the space
$\mathcal{H}$ of spin graph functions.
The resulting homomorphic image 
of the inclusions \eqref{inclusionalgintro} of algebras 
gives rise to an inclusion
\begin{equation}\label{inclH}
I\subseteq J\subseteq A
\end{equation}
of subalgebras of $\textup{End}(\mathcal{H})$. We view \eqref{inclH} as a quantum spin system with quantum state space $\mathcal{H}$, and with the homomorphic image 
$A=A_{\Gamma,G,S}\subseteq\textup{End}(\mathcal{H})$ of $\mathcal{D}(G^E)^{\mathbf{K}}$ as algebra of 
 {\it quantum observables},
the homomorphic image $J=J_{\Gamma,G,S}$ of $\mathcal{D}_{\textup{inv}}(G^E)^{\mathbf{K}}$ as algebra of {\it quantum integrals},
and the homomorphic image $I=I_{\Gamma,G,S}$ of $\mathcal{D}_{\textup{biinv}}(G^E)$ as the algebra of {\it quantum Hamiltonians}.

\subsection{}\label{1-5}
The quantum state space $\mathcal{H}$ breaks up in an algebraic direct sum of finite dimensional simultaneous eigenspaces for the action of the quantum Hamiltonians,
\[
\mathcal{H}=\bigoplus_{\chi\in I^\wedge}\mathcal{H}_\chi,
\]
with $I^\wedge$ the set of characters
 of $I$. We say that the quantum spin system is {\it superintegrable} if for all $\chi\in I^\wedge$, the simultaneous eigenspace $\mathcal{H}_\chi$ is either $\{0\}$
or an irreducible $J$-module. Equivalently, the quantum spin system is superintegrability when eigenstates $f,g\in\mathcal{H}_\chi$ with the same energy eigenvalues $\chi\in I^\wedge$ are related by a quantum integral: $g=D(f)$ for some $D\in J$.

We say that the superintegrable quantum spin system is integrable when $I=J$.
In this case $\textup{dim}(\mathcal{H}_{\chi})\leq 1$ for all $\chi\in I^\wedge$, i.e., an eigenstate is determined by its energy eigenvalues up to normalisation.

\subsection{}\label{1-6}
The main result of this paper is as follows.
\begin{thm}
The quantum spin system on $\mathcal{H}=\mathcal{H}_{\Gamma,G,S}$ is superintegrable when the following three conditions are satisfied:
\begin{enumerate}
\item[{\textup{(a)}}] $G$ is simply connected,
\item[{\textup{(b)}}] the local gauge group $K_v$ is a closed connected subgroup of $G$ for all $v\in V$,
\item[{\textup{(c)}}] $\sigma: \mathbf{K}\rightarrow\textup{GL}(S)$ is irreducible. 
\end{enumerate}
\end{thm}

We will prove this theorem using a result of Harish-Chandra \cite{HC} and Lepowsky \& McCollum \cite{LM} relating irreducible $\mathfrak{g}$-modules to irreducible $U(\mathfrak{g})^K$-modules for appropriate compact Lie groups $K$ (this result plays an important role in the proof of the subquotient theorem for Harish-Chandra modules).

\subsection{}\label{1-7}
We will say that a spin graph function $f$ is {\it elementary} if it is a simultaneous eigenfunction of the quantum Hamiltonians, i.e., when $f\in\mathcal{H}_\chi$ for some $\chi\in I^\wedge$. For tensor product $\mathbf{K}$-representations $S$ we construct spanning sets of $\mathcal{H}_\chi$ 
using the data of the following colored version of $\Gamma$. 

The colors at the vertices $v\in V$ are the local representations $\sigma_v: K_v\rightarrow\textup{GL}(S_v)$ of the tensor product 
representation $S$.
To obtain the colors of the edges, we use the fact that $\mathcal{H}_\chi\not=\{0\}$ if and only if $\chi$ is the central character of an irreducible representation of $G^E$.
The irreducible $G^E$-representation provides the colors of the edges of $\Gamma$ by local irreducible $G$-representations.

We construct the spanning set of the elementary spin graph functions in $\mathcal{H}_\chi$ in terms of
local invariant tensors (local 
in the sense that they only depend on the star of some vertex $v$ of the colored graph $\Gamma$). 

\subsection{}\label{1-8}
If $\Gamma$ is the directed cycle graph with $n$ edges with $K_v=G$ for all $v\in V$, then we show that the resulting elementary spin graph functions are essentially the generalised (or $n$-point) trace functions from Etingof \& Schiffmann \cite{ES}. 

If $\Gamma$ is the linearly ordered linear graph with $n$ edges
and the local gauge groups $K_v$ are $G$ (resp. $K$) for $2$-valent (resp. $1$-valent) vertices $v\in V$,
then we show that the resulting elementary spin graph functions are the $n$-point spherical functions from \cite{RS}. 
For $n=1$ these are the usual elementary $\sigma$-spherical functions on $G$, see, e.g., \cite{HC0,SR}. 

In both cases the local invariant tensors may be identified with topological degenerations of vertex operators, cf. \cite{SR,St}.

\subsection{}\label{1-9}
The explicit desciption of the elementary spin graph functions as multipoint trace functions and multipoint spherical functions for 
the two special cases in \S\ref{1-8}, connects the associated quantum spin systems to the {\it periodic} and {\it open quantum spin Calogero-Moser chains} from \cite{ES,RS,Re3} and
\cite{SR,RS}, respectively. 

This can be made concrete on the level of quantum Hamiltonians. It requires
a parametrisation of the moduli space $G^E/\mathbf{K}$ of $G$-graph connections in terms of an appropriate subtorus $T$ of $G$, as well as Harish-Chandra's radial component techniques to describe the action of the edge-coordinate quadratic Casimirs on $\mathcal{H}_{\Gamma,G,S}$ in terms of explicit second-order $\textup{End}(S)$-valued differential operators $H_e$ ($e\in E$) on $T$ (which are the quadratic Hamiltonians of the quantum Calogero-Moser spin chain up to a gauge). 
The
differences $H_e-H_{e^\prime}$ for neighboring edges $e$ and $e^\prime$ then form 
  an explicit commuting family of first order differential operators, called asymptotic Knizhnik-Zamolodchikov-Bernard operators. 
 See \cite{ES} and \cite{SR,RS} for the details.

\subsection{}\label{1-10}
Combining the results from \S\ref{1-6} and \S\ref{1-9} we obtain explicit conditions ensuring the superintegrability of the periodic and open quantum spin Calogero-Moser
chains. For the special case of the directed cycle graph $\Gamma$ with one vertex and $\mathbf{K}=G$, the superintegrability of 
the associated periodic quantum spin Calogero-Moser system
was considered before in \cite{Re2}. The classical superintegrability of the periodic and open Calogero-Moser spin chains is discussed in \cite{Re23}.

\subsection{}\label{1-11}
The contents of the paper is as follows.

In {\it Section \ref{2}} we describe the type of quantum systems that we consider in this paper, and discuss the concept of superintegrability in this context. 

In {\it Section \ref{3}} we formulate a result of Harish-Chandra \cite{HC} and Lepowsky \& McCollum \cite{LM} (Corollary \ref{3-17}) that will play the key role in establishing the superintegrability of the quantum spin systems on moduli spaces of graph connections. This involves the concept of reductive extensions of Lie algebras, which we discuss in detail.

In {\it Section \ref{4}} we introduce the space of spin graph functions on graph connections, and provide a spanning set in terms of local invariant tensors (Theorem \ref{4-15}). For 
the directed cycle graph and the linearly ordered linear graph we relate the spin graph functions to multipoint trace functions and multipoint spherical functions (see \S\ref{4-17} and \S\ref{4-18}).
 
 In {\it Section \ref{5}} we provide a detailed introduction of the quantum spin systems on moduli spaces of graph connections. We state the conditions ensuring superintegrability of the quantum spin system 
  and discuss the superintegrability of the quantum periodic and open spin Calogero-Moser chains (see \S\ref{5-12} and \S\ref{5-13}).
 
In {\it Section \ref{6}} we give the proof of the main result (Theorem \ref{1-6}/\ref{5-11}). The crucial intermediate step, which will allow us to use the result of Harish-Chandra and Lepowsky \& McCollum in this context, is the translation of the condition of superintegrability in terms of irreducibility conditions of local intertwining spaces at the stars of the vertices of the graph (see Corollary \ref{6-9}).

\vspace{1cm}
\noindent
{\bf Conventions:} The ground field will be $\mathbb{C}$ unless explicitly stated otherwise. Lie algebras are finite dimensional unless stated explicitly otherwise. We use $\textup{Hom}(V,W)$ for the Hom-space in the category of complex vector spaces. For $G$ a group, $A$ an associative algebra and $\mathfrak{g}$ a Lie algebra
we write $\textup{Hom}_G(V,W)$, $\textup{Hom}_A(V,W)$, $\textup{Hom}_{\mathfrak{g}}(V,W)$ for the Hom-space in the category of $G$-representations, left $A$-modules and $\mathfrak{g}$-modules, respectively. We write $U(\mathfrak{k})$ for the universal enveloping algebra of a complex Lie algebra $\mathfrak{k}$, and $Z(\mathfrak{k})$ for its center.

For sets $X,\mathcal{I}$ with $\mathcal{I}$ finite, we write $X^{\mathcal{I}}$ for the direct product of $\#\mathcal{I}$-copies of $X$. In case $X$ is a Lie group/algebra, we endow $X^{\mathcal{I}}$ with the direct
product Lie group/algebra structure.

For a finite family $\{M_i\}_{i\in\mathcal{I}}$ of vector spaces $M_i$ with index set $\mathcal{I}=\{i_1,\ldots,i_s\}$, totally ordered by $i_1<\cdots<i_s$, we write 
\[
\bigotimes_{i\in\mathcal{I}}M_i:=M_{i_1}\otimes\cdots\otimes M_{i_s}.
\]
\vspace{.5cm}\\
\noindent
{\bf Acknowledgements:} both authors were supported by the Dutch Research Council (NWO), project number 613.009.1260. In addition, the work of N.R. was supported by the NSF grant DMS-1902226, by the RSF grant 18-11-00-297 and by the Changjiang fund.

\section{Centraliser algebras}\label{2}
In this section we derive some elementary properties of centraliser algebras, with an eye towards the application to quantum superintegrable systems.
The starting point is a complex vector space $\mathcal{H}$ and an inclusion
\[
I\subseteq A\subseteq \textup{End}(\mathcal{H})
\]
of unital algebras, with $I$ being commutative. In applications to quantum mechanics $\mathcal{H}$ is the quantum state space,
$A$ the algebra of observables, and $I$ its subalgebra of quantum Hamiltonians. We do not fix a particular $H\in I$ as 
the quantum Hamiltonian of the system, since we are not considering quantum dynamics at this point.

\subsection{}\label{2-1} 
Denote by $I^{\wedge{}}$ the set of characters of $I$. For an element $\chi\in I^{\wedge{}}$, i.e., for an unital algebra homomorphism $\chi: I \rightarrow\mathbb{C}$, we write
\[
\mathcal{H}_\chi:=\{h\in\mathcal{H}\,\, | \,\, y\cdot h=\chi(y)h\quad \forall\, y\in I\}
\]
for the corresponding joint eigenspace (it may be zero).
\subsection{}\label{2-2} 
Denote by
\[
C_A(I):=\{x\in A \,\, | \,\, xy=yx\quad \forall y\in I\}
\]
the centraliser of $I$ in $A$. It is a subalgebra of $A$ containing $I$. It stabilises $\mathcal{H}_\chi$ for all $\chi\in I^\wedge$.
\subsection{}\label{2-3} 
Suppose that $J\subseteq\textup{End}(\mathcal{H})$ is a sub-algebra stabilising $\mathcal{H}_\chi\subseteq\mathcal{H}$.
Then $\mathcal{H}_\chi$ is a $J$-module, and 
\[
J_\chi:=\{x\vert_{\mathcal{H}_\chi}\,\, | \,\, x\in J\}
\]
is a sub-algebra of $\textup{End}(\mathcal{H}_\chi)$. If $\mathcal{H}_\chi$ is a finite dimensional irreducible $J$-module, then
 $J_\chi=\textup{End}(\mathcal{H}_\chi)$ by the density theorem. If $J=I$ then we have $I_\chi=\mathbb{C}\textup{id}_{\mathcal{H}_\chi}$.
\subsection{}\label{2-4} 
Let $J$ be a sub-algebra of $A$ containing $I$. Then 
\[
C_A(J)\subseteq C_A(I).
\]
If in addition $J$ stabilises $\mathcal{H}_\chi$ for some $\chi\in I^\wedge$, then
$C_A(J)$ stabilises $\mathcal{H}_\chi$
in view of \S\ref{2-2}. The fact that both $J$ and $C_A(J)$ stabilise $\mathcal{H}_\chi$ implies that $C_A(J)_\chi$ is contained in the commutant of $J_\chi$ in $\textup{End}(\mathcal{H}_\chi)$. 

If in addition $\mathcal{H}_\chi$ is an irreducible finite dimensional $J$-module (in particular, $\mathcal{H}_\chi\not=\{0\}$),
then 
\[
C_A(J)_\chi=\mathbb{C}\,\textup{id}_{\mathcal{H}_\chi}=I_\chi
\]
by Schur's lemma.
\subsection{}\label{2-5}
Suppose that $J\subseteq A$ is a subalgebra satisfying 
\[
I\subseteq J\subseteq C_A(I).
\]
For such an algebra $J$ the joint eigenspace $\mathcal{H}_\chi$ is $J$-stable for all $\chi\in I^\wedge$, in view of \S\ref{2-2}.

If in addition $\mathcal{H}$ is a semisimple $I$-module (i.e., $\mathcal{H}=\bigoplus_{\chi\in I^\wedge}\mathcal{H}_\chi$),
then the map 
\[
J\rightarrow\prod_{\chi\in I^\wedge}\textup{End}(\mathcal{H}_\chi),\qquad x\mapsto \bigl(x\vert_{\mathcal{H}_\chi}\bigr)_{\chi\in I^\wedge}
\]
is an injective algebra homomorphism, with $\prod_{\chi\in I^\wedge}\textup{End}(\mathcal{H}_\chi)$ the direct product of the
family $\{\textup{End}(\mathcal{H}_\chi) \,\, | \,\, \chi\in I^\wedge\}$ of algebras. Its image is contained in $\prod_{\chi\in I^\wedge}J_\chi$.
\subsection{}\label{2-6}
Suppose that $J\subseteq A$ is a subalgebra satisfying
\begin{equation}\label{qssweak}
I\subseteq J\subseteq C_A(I)\subseteq A\subseteq\textup{End}(\mathcal{H}).
\end{equation}
Assume furthermore that the following two spectral properties hold true:
\begin{enumerate}
\item[(a)] $\mathcal{H}$ is a semisimple $I$-module.
\item[(b)] For $\chi\in I^\wedge$, either $\mathcal{H}_\chi=\{0\}$ or $\mathcal{H}_\chi$ is an irreducible finite dimensional $J$-module.
\end{enumerate}
By \S\ref{2-3} and \S\ref{2-4} we then have 
\begin{equation}\label{qsprop}
\begin{split}
J_\chi&=\textup{End}(\mathcal{H}_\chi)=C_A(I)_\chi,\\
I_\chi&=\mathbb{C}\textup{id}_{\mathcal{H}_\chi}=C_A(J)_\chi
\end{split}
\end{equation}
for all $\chi\in I^\wedge$. Informally speaking, $J$ is ``locally'' of maximal size and equal to $C_A(I)$, and $I$ is ``locally'' the center of $J$.
\subsection{}\label{2-7}
The setup of \S\ref{2-6} provides the mathematical framework for superintegrability of quantum systems in this paper. From this perspective \eqref{qssweak} is defining a quantum system with quantum state space $\mathcal{H}$, algebra of quantum observables $A$, algebra of quantum Hamiltonians 
$I$ and algebra of quantum integrals $J$.
 \begin{defi}
 The quantum system \eqref{qssweak} is said to be superintegrable if the two spectral conditions \ref{2-6}(a)\&(b) hold true. 
 \end{defi}
The resulting properties \eqref{qsprop} for the quantum superintegrable system provide the link with the notion of a core structure of a quantum superintegrable system considered in \cite[\S 2]{Re2}.

A quantum superintegrable system is said to be quantum integrable if $I=J$.
In this case $\textup{dim}(\mathcal{H}_\chi)\leq 1$ for all $\chi\in I^\wedge$, i.e., the eigenvalues of the quantum Hamiltonians determine the 
corresponding joint eigenvector up to a multiplicative constant. 

For quantum superintegrable systems eigenstates this is no longer true. 
But when $f,g\in\mathcal{H}_\chi$ then there exists a quantum integral $D\in J$ such that $g=D(f)$. Here we use that by the density theorem, the irreducibility of the $J$-module $\mathcal{H}_\chi$ is equivalent to  
\[
J_\chi=\textup{End}(\mathcal{H}_\chi).
\]

For the examples of quantum superintegrable systems we thus have the weaker condition that simultaneous eigen\-spa\-ces are finite dimensional, but two joint 
eigenvectors
with the same eigenvalues can always be related through the action of a quantum integral.  
\subsection{}\label{2-8}
From the perspective of quantisation, quantum superintegrability requires the algebras $I,J$ and $A$ to be quantisations of the Poisson algebras of Hamiltonians, integrals and observables for a classical superintegrable system (which is also sometimes called a degenerate integrable system). This is known in the case of periodic and open quantum spin Calogero-Moser chains \cite{Re1,Re23}. 
For a discussion of classical superintegrability, see \cite{Re1} and references therein. 

\section{Preservation of irreducibility}\label{3}
In this section we focus on a purely representation theoretic result due to Lepowsky and McCollum \cite[Thm. 5.5]{LM} (in special cases it goes back to Harish-Chandra \cite[Thm. 2]{HC}). It will be the crucial ingredient in proving superintegrability of the quantum spin systems on graph connections in Section \ref{6}.
\subsection{}\label{3-1}
Let $\mathfrak{g}$ be a Lie algebra. Recall that a $\mathfrak{g}$-module $M$ is said to be semisimple if $M$ is the sum of its irreducible $\mathfrak{g}$-submodules.
If furthermore all the irreducible $\mathfrak{g}$-submodules of $M$ are finite dimensional, then we say that $M$ is a {\it finitely semisimple} $\mathfrak{g}$-module. 

\subsection{}\label{3-2}

Let $G$ be a real Lie group and $K\subseteq G$ a connected compact Lie subgroup. Denote by $\mathfrak{g}_0$ the Lie algebra of $G$,
and $\mathfrak{g}$ its complexification.  If $\pi$ is a Hilbert space representation, then its (dense) subspace $M$ of smooth $K$-finite vectors becomes a $(\mathfrak{g},K)$-module with $x\in\mathfrak{g}_0$ acting by 
\[
x\cdot m:=\frac{d}{dt}\bigg|_{t=0}\pi(\exp(tx))m
\]
(see, e.g., \cite[\S 3.3.1]{Wa} for the definition of a $(\mathfrak{g},K)$-module).
The $(\mathfrak{g},K)$-module $M$
is finitely semisimple as a $\mathfrak{k}$-module.
Furthermore, if 
$\pi$ is irreducible and unitary,
 then the associated $(\mathfrak{g},K)$-module $M$ is irreducible as $\mathfrak{g}$-module. This in fact holds true under the weaker assumption that $\pi$ is irreducible and admissible
(see, e.g., \cite[\S 3.3-4]{Wa} for further details). 

\subsection{}\label{3-3}
Let $\mathfrak{k}\subseteq\mathfrak{g}$ be an inclusion of Lie algebras and $M$ a $\mathfrak{g}$-module. Denote by $\mathfrak{k}^\wedge$ the set of isomorphism classes of finite dimensional irreducible $\mathfrak{k}$-modules. For $\alpha\in\mathfrak{k}^\wedge$ the {\it $\alpha$-isotypical component $M_\alpha$ of $M$} is the subspace of $M$ generated by the finite dimensional irreducible $\mathfrak{k}$-submodules of $M$ from the isomorphism class $\alpha$. The sum
$\sum_{\alpha\in\mathfrak{k}^\wedge}M_\alpha\subseteq M$
is direct (see, e.g., \cite[\S 1.2.8]{Di}). Furthermore, $M=\bigoplus_{\alpha\in\mathfrak{k}^\wedge}M_\alpha$ if and only if $M$ is finitely semisimple as a $\mathfrak{k}$-module.
\subsection{}\label{3-4}
A Lie subalgebra $\mathfrak{k}\subseteq\mathfrak{g}$ is said to be {\it reductive in $\mathfrak{g}$} when $\mathfrak{g}$ is a
semisimple $\textup{ad}(\mathfrak{k})$-module. 

Note that if $\mathfrak{k}$ is reductive in $\mathfrak{g}$, then $\mathfrak{k}$ is a reductive Lie algebra. 
On the other hand, if $\mathfrak{k}$ is a semisimple Lie subalgebra of $\mathfrak{g}$, then $\mathfrak{k}$ is reductive in $\mathfrak{g}$ by
Weyl's complete reducibility theorem. 

\subsection{}\label{3-5}
Let $G$ be a real Lie group with Lie algebra $\mathfrak{g}_0$, and $K\subseteq G$ a connected compact Lie subgroup. Denote by 
$\mathfrak{k}$ and $\mathfrak{g}$ the complexified Lie algebras of $K$ and $G$, respectively. Then $\mathfrak{k}$ is reductive in $\mathfrak{g}$. 

\subsection{}\label{3-6}
Let $\mathfrak{g}$ be a Lie algebra and $\theta\in\textup{Aut}(\mathfrak{g})$ an automorphism of finite order $n$. The associated fix-point Lie algebra is
\[
\mathfrak{g}^\theta:=\{x\in\mathfrak{g} \,\, | \,\, \theta(x)=x\}.
\]
\begin{prop}
Suppose that $\theta$ is an automorphism of a semisimple Lie algebra $\mathfrak{g}$ of finite order $m$. Then $\mathfrak{g}^\theta$ is reductive in $\mathfrak{g}$.
\end{prop}
\begin{proof}
The proof is a rather straightforward adjustment of the proof of the statement for involutions, see \cite[Prop. 1.13.3]{Di}.
We give the proof for convenience of the reader.

Denote by $\mathfrak{g}_{\overline{r}}$ ($\overline{r}\in\mathbb{Z}/m\mathbb{Z}$) the eigenspace of $\theta$ with eigenvalue $e^{2\pi i/r}$. 
Then $\mathfrak{g}^\theta=\mathfrak{g}_{\overline{0}}$ and $\mathfrak{g}_{\overline{r}}$ are $\textup{ad}(\mathfrak{g}^\theta)$-invariant subspaces of $\mathfrak{g}$. 
Since $\theta$ is of order $m$, the assignment $\overline{1}\mapsto\theta$ defines a representation $\mathbb{Z}/m\mathbb{Z}\rightarrow\textup{GL}(\mathfrak{g})$ of the finite abelian group $\mathbb{Z}/m\mathbb{Z}$. By Maschke's theorem, we conclude that
\[
\mathfrak{g}=\bigoplus_{\overline{r}\in\mathbb{Z}/m\mathbb{Z}}\mathfrak{g}_{\overline{r}}.
\]

Write $\mathfrak{p}:=\bigoplus_{\overline{r}\not=\overline{0}}\mathfrak{g}_{\overline{r}}$, so that
\[
\mathfrak{g}=\mathfrak{g}^\theta\oplus\mathfrak{p}.
\]

Let $\kappa: \mathfrak{g}\times\mathfrak{g}\rightarrow\mathbb{C}$ be the Killing form of $\mathfrak{g}$. Then $\kappa(\theta(x),\theta(y))=\kappa(x,y)$ for all $x,y\in\mathfrak{g}$, hence
$\kappa(\mathfrak{g}^\theta,\mathfrak{p})=0$. 
Since $\mathfrak{g}$ is semisimple, we conclude that $\kappa\vert_{\mathfrak{g}^\theta\times\mathfrak{g}^\theta}$ is nondegenerate. Furthermore, if $x\in\mathfrak{g}^\theta$ and $x=s+n$ is the abstract Chevalley decomposition of $x$ in $\mathfrak{g}$, with $s\in\mathfrak{g}$ (resp. $n\in\mathfrak{g}$) the semisimple (resp. nilpotent) part of $x$, then
$s,n\in\mathfrak{g}^\theta$ (this holds true for any automorphism $\theta$ of $\mathfrak{g}$). Then \cite[Prop. 1.7.6]{Di} implies that $\mathfrak{g}^\theta$ is reductive in $\mathfrak{g}$.
\end{proof}

\subsection{}\label{3-7}
If $\mathfrak{k}$ is reductive in $\mathfrak{g}$ and $M$ is a finitely semisimple $\mathfrak{g}$-module, then $M$ is finitely semisimple as a $\mathfrak{k}$-module by \cite[Prop. 1.7.9(ii)]{Di}. In particular, suppose that we have inclusions of Lie algebras
\[
\mathfrak{l}\subseteq\mathfrak{m}\subseteq\mathfrak{g}
\]
where $\mathfrak{m}$ is reductive in $\mathfrak{g}$ and $\mathfrak{l}$ is reductive in $\mathfrak{m}$, then $\mathfrak{l}$ is reductive in $\mathfrak{g}$.

\subsection{}\label{3-8}
For a homomorphic image of a Lie algebra $\mathfrak{k}\subseteq\mathfrak{g}$ which is reductive in $\mathfrak{g}$, we have the following result.
\begin{lem}
Suppose that $\mathfrak{k}$ is reductive in $\mathfrak{g}$. Let $\phi: \mathfrak{g}\twoheadrightarrow\mathfrak{l}$ be an epimorphism of Lie algebras. Then
$\phi(\mathfrak{k})$ is reductive in $\mathfrak{l}$.
\end{lem}
\begin{proof}
Let
$\mathfrak{g}=\bigoplus_{i=1}^mS_i$ be a decomposition as a direct sum of finite dimensional irreducible $\textup{ad}(\mathfrak{k})$-modules.
Then $\mathfrak{l}=\sum_{i=1}^m\phi(S_i)$. Either $\phi(S_i)=\{0\}$ or $\phi(S_i)$ is an irreducible $\textup{ad}(\phi(\mathfrak{k}))$-module. By a straightforward induction argument it follows that $\mathfrak{l}=\bigoplus_{i\in\mathcal{I}}\phi(S_i)$ for some subset $\mathcal{I}\subseteq \{1,\ldots,m\}$. This completes the proof.
\end{proof}

\subsection{}\label{3-9}
For $m\in\mathbb{Z}_{>0}$ denote by $\delta_{\mathfrak{g}}^{(m)}: \mathfrak{g}\rightarrow\mathfrak{g}^{\times m}$ the Lie algebra homomorphism mapping $x\in\mathfrak{g}$ to the $m$-tuple $(x,\ldots,x)$. If $\mathfrak{k}$ is a Lie subalgebra of $\mathfrak{g}$, then we denote by
$\mathfrak{k}^{(m)}\subseteq\mathfrak{g}^{\times m}$ its image under $\delta_{\mathfrak{g}}^{(m)}$.
\begin{prop}
Suppose that $\mathfrak{g}$ is semisimple and that $\mathfrak{k}$ is reductive in $\mathfrak{g}$. Then $\mathfrak{k}^{(m)}$ is reductive in $\mathfrak{g}^{\times m}$.
\end{prop}
\begin{proof}
By Lemma \ref{3-8}, $\mathfrak{k}^{(m)}$ is reductive in $\mathfrak{g}^{(m)}$. Note that
\[
\mathfrak{g}^{(m)}=(\mathfrak{g}^{\times m})^{\theta_m}
\]
with $\theta_m$ the automorphism of $\mathfrak{g}^{\times m}$ of order $m$ defined by 
\[
\theta_m(x_1,\ldots,x_m):=(x_m,x_1,\ldots,x_{m-1}).
\]
Proposition \ref{3-6} then shows that $\mathfrak{g}^{(m)}$ is reductive in $\mathfrak{g}^{\times m}$. Hence $\mathfrak{k}^{(m)}$ is reductive in $\mathfrak{g}^{\times m}$ by \S\ref{3-7}.
\end{proof}

\subsection{}\label{3-10}
Let $\mathfrak{k}\subseteq\mathfrak{g}$ be an inclusion of Lie algebras.
We say that $\mathfrak{g}$ is a {\it reductive extension of $\mathfrak{k}$} 
when the inclusion map $\mathfrak{k}\hookrightarrow\mathfrak{g}$ is a section of $\textup{ad}(\mathfrak{k})$-modules and the quotient module $\mathfrak{g}/\mathfrak{k}$ is a semisimple $\mathfrak{k}$-module.
We then typically write $\mathfrak{p}$ for a choice of an $\textup{ad}(\mathfrak{k})$-invariant complement of $\mathfrak{k}$ in $\mathfrak{g}$ (which is finitely semisimple as $\textup{ad}(\mathfrak{k})$-module).  
\subsection{}\label{3-11}
Let $\mathfrak{k}\subseteq\mathfrak{g}$ be a Lie subalgebra. 
The following two statements are equivalent:
\begin{enumerate}
\item[(a)] $\mathfrak{k}$ is reductive in $\mathfrak{g}$.
\item[(b)] $\mathfrak{k}$ is a reductive Lie algebra and $\mathfrak{g}$ is a reductive extension of $\mathfrak{k}$.
\end{enumerate}
In particular, \S\ref{3-5}, \S\ref{3-6} and \S\ref{3-7} provide examples of reductive extensions.

\subsection{}\label{3-12} 
The following result should be compared to the transitivity property in \S\ref{3-7}.
\begin{lem}
Let $\mathfrak{l}\subseteq\mathfrak{m}\subseteq\mathfrak{g}$ be inclusions of finite dimensional Lie algebras. Suppose that $\mathfrak{g}$ is a reductive extension of $\mathfrak{m}$
and that $\mathfrak{l}$ is reductive in $\mathfrak{m}$. Then $\mathfrak{l}$ is reductive in $\mathfrak{g}$. 
\end{lem}
\begin{proof}
Let $\mathfrak{p}\subseteq \mathfrak{g}$ be an $\textup{ad}(\mathfrak{m})$-invariant subspace such that $\mathfrak{g}=\mathfrak{m}\oplus\mathfrak{p}$. By the assumptions, $\mathfrak{m}$ is a finite dimensional semisimple $\textup{ad}(\mathfrak{l})$-module and $\mathfrak{p}$ is a finite dimensional semisimple $\textup{ad}(\mathfrak{m})$-module. It then follows from  \cite[Prop. 1.7.9(ii)]{Di} (see also \S\ref{3-7}) that
$\mathfrak{p}$ is also semisimple as an $\textup{ad}(\mathfrak{l})$-module. Since $\mathfrak{l}$ is reductive in $\mathfrak{m}$, it follows that $\mathfrak{l}$ is also reductive in $\mathfrak{g}$. 
\end{proof}
\subsection{}\label{3-13}
The following lemma is the analog of Lemma \ref{3-8} for reductive extensions.
\begin{lem}
Let $\mathfrak{g}$ be a reductive extension of $\mathfrak{k}$. Let $\phi: \mathfrak{g}\twoheadrightarrow\mathfrak{l}$ be an epimorphism of Lie algebras. Then
$\mathfrak{l}$ is a reductive extension of $\phi(\mathfrak{k})$.
\end{lem}
\begin{proof}
Let $\mathfrak{p}\subseteq\mathfrak{g}$ be an $\textup{ad}(\mathfrak{k})$-invariant subspace such that $\mathfrak{g}=\mathfrak{k}\oplus\mathfrak{p}$. Let
$\mathfrak{p}=\bigoplus_{i=1}^mS_i$ be a decomposition as a direct sum of finite dimensional irreducible $\textup{ad}(\mathfrak{k})$-modules.
Then $\mathfrak{l}=\phi(\mathfrak{k})+\sum_{i=1}^m\phi(S_i)$, and either $\phi(S_i)=\{0\}$ or $\phi(S_i)$ is an irreducible $\textup{ad}(\phi(\mathfrak{k}))$-module. 
A straightforward induction argument then shows that $\mathfrak{l}=\phi(\mathfrak{k})\oplus\bigoplus_{i\in\mathcal{I}}\phi(S_i)$ for some subset $\mathcal{I}\subseteq \{1,\ldots,m\}$. This completes the proof.
\end{proof}

\subsection{}\label{3-14}
Let $\mathfrak{k}\subseteq\mathfrak{g}$ be an inclusion of Lie algebras and 
$\mathfrak{g}$ a reductive extension of $\mathfrak{k}$. Lepowsky and McCollum \cite[Prop. 4.2]{LM} obtained the following criterion to detect whether an irreducible $\mathfrak{g}$-module is finitely semisimple as a $\mathfrak{k}$-module.
\begin{prop}
Let $\mathfrak{g}$ be a reductive extension of $\mathfrak{k}$, and 
$M$ an irreducible $\mathfrak{g}$-module. Then $M$ is finitely semisimple as a $\mathfrak{k}$-module unless $M_\alpha=0$ for all $\alpha\in\mathfrak{k}^\wedge$. 
\end{prop}

\subsection{}\label{3-15}
Let $\mathfrak{k}\subseteq\mathfrak{g}$ be an inclusion of Lie algebras. Using the associated canonical inclusion $U(\mathfrak{k})\subseteq U(\mathfrak{g})$ of universal enveloping algebras, we set
\[
U(\mathfrak{g})^{\mathfrak{k}}:=C_{U(\mathfrak{g})}\bigl(U(\mathfrak{k})\bigr)
\]
for the centraliser subalgebra of $U(\mathfrak{k})$ in $U(\mathfrak{g})$ (which equals the centraliser of $\mathfrak{k}$ in $U(\mathfrak{g})$).

Let $M$ be a $\mathfrak{g}$-module, and view it as an $U(\mathfrak{g})$-module. The corresponding homomorphism
$U(\mathfrak{g})\rightarrow \textup{End}(M)$
restricts to an algebra map 
\[
U(\mathfrak{g})^{\mathfrak{k}}\rightarrow\textup{End}_{\mathfrak{k}}(M).
\] 
As a consequence, for a $\mathfrak{k}$-module $S$ 
the space $\textup{Hom}_{\mathfrak{k}}(S,M)$
of $\mathfrak{k}$-linear maps $S\rightarrow M$ becomes a left $U(\mathfrak{g})^{\mathfrak{k}}$-module, with $U(\mathfrak{g})^{\mathfrak{k}}$ acting on the codomain $M$.
\subsection{}\label{3-16}
Let $\mathfrak{k}\subseteq\mathfrak{g}$ be an inclusion of Lie algebras. Let $S^{\alpha}$ be a finite dimensional irreducible $\mathfrak{k}$-module of isomorphism class $\alpha\in\mathfrak{k}^\wedge$.
For a $\mathfrak{g}$-module $M$, the space
\[
\textup{Hom}_{\mathfrak{k}}\bigl(S^{\alpha},M)
\]
models the multiplicity space of $S^{\alpha}$ in $M$.
In fact, $\textup{Hom}_{\mathfrak{k}}\bigl(S^{\alpha},M)$ is isomorphic to $\textup{Hom}_{\mathfrak{k}}\bigl(S^{\alpha},M_\alpha)$ as a complex vector space, and the $\mathfrak{k}$-module $M_\alpha$ is isomorphic to an algebraic direct sum of $\textup{dim}(\textup{Hom}_{\mathfrak{k}}(S^\alpha,M))$ copies of $S^{\alpha}$ (see \cite[\S 1.2.8]{Di}). 
In particular, $\textup{Hom}_{\mathfrak{k}}\bigl(S^{\alpha},M)=0$ if and only if $M_\alpha=0$. Hence Proposition \ref{3-14} can be restated as follows:
\begin{prop}
Let $\mathfrak{g}$ be a reductive extension of $\mathfrak{k}$, and $M$ an irreducible $\mathfrak{g}$-module. Then $M$ is finitely semisimple as a $\mathfrak{k}$-module unless $\textup{Hom}_{\mathfrak{k}}(S^\alpha,M)=0$ for all $\alpha\in\mathfrak{k}^\wedge$. 
\end{prop}
The multiplicity space $\textup{Hom}_{\mathfrak{k}}\bigl(S^{\alpha},M)$ ``remembers'' the $\mathfrak{g}$-action on $M$ through the left $U(\mathfrak{g})^{\mathfrak{k}}$-action from \S\ref{3-15}. Up to isomorphism of $U(\mathfrak{g})^{\mathfrak{k}}$-modules, the multiplicity space $\textup{Hom}_{\mathfrak{k}}(S^{\alpha},M)$ does not depend on the choice of $S^{\alpha}$.

\subsection{}\label{3-17}
Let $\mathfrak{g}$ be a reductive extension of $\mathfrak{k}$ and $\alpha\in\mathfrak{k}^\wedge$ an isomorphism class of a finite dimensional irreducible $\mathfrak{k}$-module. 
Lepowsky and McCollum \cite[Thm. 5.5]{LM} showed that $M\mapsto \textup{Hom}_{\mathfrak{k}}(S^{\alpha},M)$ gives rise to a bijective correspondence between the isomorphism classes of irreducible $\mathfrak{g}$-modules $M$ with $M_\alpha\not=0$ and the isomorphism classes of irreducible modules over the quotient algebra $U(\mathfrak{g})^{\mathfrak{k}}/(U(\mathfrak{g)}^{\mathfrak{k}}\cap U(\mathfrak{g})\mathcal{J}^\alpha)$,
where $\mathcal{J}^\alpha\subseteq U(\mathfrak{k})$ is the annihilator of $S^{\alpha}$ in $U(\mathfrak{k})$. In the context of \S\ref{3-2}, this correspondence goes back to Harish-Chandra
\cite{HC}.

In view of \S\ref{3-16}, we have the following immediate consequence of this correspondence.
\begin{cor}
Let $\mathfrak{g}$ be a reductive extension of $\mathfrak{k}$. Let $M$ be an irreducible $\mathfrak{g}$-module. 

For each $\alpha\in\mathfrak{k}^\wedge$, the multiplicity space $\textup{Hom}_{\mathfrak{k}}(S^{\alpha},M)$ is either $\{0\}$ or it is an irreducible
$U(\mathfrak{g})^{\mathfrak{k}}$-module.
\end{cor}

\section{Spin graph functions}\label{4}
In this section we introduce the space $\mathcal{H}=\mathcal{H}_{G,\Gamma,S}$ of spin graph functions and construct spanning sets
of $\mathcal{H}$ using local tensor invariants. Here $G$ is a connected compact Lie group, and $\Gamma=(V,E,s,t)$ is a finite oriented graph with vertices $V=\{v_1,\ldots,v_r\}$, edges $E=\{e_1,\ldots,e_n\}$ and source and target maps
$s,t: E\rightarrow V$. 
\subsection{}\label{4-1}
Let $C(G)$ be the space of continuous complex-valued functions on $G$, viewed as $G^{\times 2}$-representation by the left-and right-regular $G$-action,
\[
\bigl((g_1,g_2)\cdot f\bigr)(g):=f(g_1^{-1}gg_2).
\]
Let $\mathcal{R}(G)\subset C(G)$ be the subalgebra of representative functions on $G$. In other words, $\mathcal{R}(G)$ consists of the functions $f\in C(G)$ which generate a finite dimensional $G^{\times 2}$-subrepresentation of $C(G)$.

\subsection{}\label{4-2}
Let $\pi: G\rightarrow \textup{GL}(M)$ be a finite dimensional continuous $G$-representation, and denote by $\pi^*: G\rightarrow\textup{GL}(M^*)$ its dual representation.
For $m\in M$ and $\phi\in M^*$ we write $c_{\phi,m}^\pi\in C(G)$ for the associated matrix coefficient
\[
c_{\phi,m}^\pi(g):=\phi\bigl(\pi(g)m\bigr)=(\pi^*(g^{-1})\phi)(m)
\]
of $\pi$ (and $\pi^*$). Then $c_{\phi,m}^\pi\in\mathcal{R}(G)$. Moreover, the space $\mathcal{R}^{\pi}(G)$ spanned by the matrix coefficients $c_{\phi,m}^\pi$ ($m\in M$, $\phi\in M^*$) is a $G^{\times 2}$-invariant subspace of $\mathcal{R}(G)$. 
In fact, the map
\[
M^*\otimes M\overset{\sim}{\longrightarrow}\mathcal{R}^{\pi}(G),\qquad \phi\otimes m\mapsto c_{\phi,m}^\pi
\]
is an isomorphism of $G^{\times 2}$-representations, where $M^*\otimes M$ is endowed with the natural tensor product action of $G^{\times 2}$.
\subsection{}\label{4-3}
Let $G^\wedge$ be the set of isomorphism classes of irreducible finite dimensional continuous $G$-representations. We denote the isomorphism class of an irreducible finite dimensional $G$-representation simply by its representation map $\pi$. The associated representation space is then denoted by $M^\pi$. 

The Peter-Weyl theorem yields the decomposition
\[
\mathcal{R}(G)=\bigoplus_{\pi\in G^\wedge}\mathcal{R}^{\pi}(G)
\]
of $\mathcal{R}(G)$ in irreducible $G^{\times 2}$-subrepresentations.

\subsection{}\label{4-4}
For two $G$-representations $M$ and $N$ we identify $M^*\otimes N^*\simeq (M\otimes N)^*$ as $G^{\times 2}$-representations, where $M\otimes N$ and $M^*\otimes N^*$ are endowed with the tensor product $G^{\times 2}$-action. Under this correspondence, $\phi\otimes\psi$ for $\phi\in M^*$ and $\psi\in N^*$ corresponds to the linear functional
on $M\otimes N$ satisfying $m\otimes n\mapsto \phi(m)\psi(n)$ for $m\in M$ and $n\in N$. In particular, if $\{m_i\}_i$ and $\{n_j\}_j$ are bases of $M$ and $N$ and $\{m_i^*\}_i$ and $\{n_j^*\}_j$ are the respective dual bases of $M^*$ and $N^*$, then the basis $\{m_i\otimes n_j\}_{i,j}$ of $M\otimes N$ has dual basis $\{m_i^*\otimes n_j^*\}_{i,j}$.

\subsection{}\label{4-5}
We write $G^E$ for the compact product Lie group $G^{E}$. Its elements are denoted either by $\bm{g}=(g_e)_{e\in E}$ 
or by $\bm{g}=(g_1,\ldots,g_n)$,
with $g_j=g_{e_j}$ the group element attached to the edge $e_j$. The group $G^E$ is called the group of {\it graph $G$-connections on $\Gamma$}.

\subsection{}\label{4-6}
For each vertex $v\in V$ we fix a Lie subgroup $K_v\subseteq G$. It will play the role as {\it local gauge group} at the vertex $v$. The product group
$\mathbf{K}:=\prod_{v\in V}K_v$ is the associated {\it gauge group}. It is a subgroup of the group $G^V$ of lattice gauge transformations.

A group element in $\mathbf{K}$ is denoted by $\bm{k}=(k_v)_{v\in V}$ with $k_v\in K_v$. 
We will sometimes write $\bm{k}=(k_1,\ldots,k_r)$ with $k_i=k_{v_i}$.
The {\it gauge action} of $\mathbf{K}$ on $G^E$ is defined by
\[
\bm{k}\cdot\bm{g}:=\bigl(k_{s(e)}g_ek_{t(e)}^{-1}\bigr)_{e\in E}\qquad\qquad \bigl(\bm{k}\in\mathbf{K},\,\,\bm{g}\in G^E\bigr).
\]
\subsection{}\label{4-7}
Let $\sigma:\mathbf{K}\rightarrow\textup{GL}(S)$ be a finite dimensional representation. The space of global algebraic sections of the associated vector bundle over 
$G^E/\mathbf{K}$ is denoted by $\mathcal{H}=\mathcal{H}_{\Gamma,G,S}$. Concretely, $\mathcal{H}$ is the space 
\[
\bigl(\mathcal{R}(G^E)\otimes S\bigr)^{\mathbf{K}}
\]
of $\mathbf{K}$-invariant $S$-valued representative functions $f$ on $G^E$, relative to the $\mathbf{K}$-action
\[
\bigl(\bm{k}\cdot f)(\bm{g}):=\sigma(\bm{k})f(\bm{k}^{-1}\cdot\bm{g})
\]
on $\mathcal{R}(G^E)\otimes S$.
In other words, $\mathcal{H}$ consists of the $\mathbf{S}$-valued representative functions $f$ on $G^E$
satisfying the equivariance property
\[
f(\bm{k}\cdot \bm{g})=\sigma(\bm{k})f(\bm{g})
\]
for $\bm{k}\in\mathbf{K}$ and $\bm{g}\in G^E$. 
We call functions $f\in\mathcal{H}$ {\it spin graph functions} (spin refers to the interpretation of $S$ as spin space for the associated quantum spin system, see \S\ref{5-9},
\S\ref{5-12} and \S\ref{5-13}).

\subsection{}\label{4-8}

Fix $\pi\in (G^E)^\wedge$ an isomorphism class of a finite dimensional irreducible representation of $G^E$ and fix a finite dimensional representation $\sigma:\mathbf{K}\rightarrow\textup{GL}(S)$. Then the subspace 
\[
\mathcal{R}^{\pi}(G^E)\otimes S\subseteq\mathcal{R}(G^E)\otimes S
\]
of {\it $\pi$-elementary spin graph functions} is $\mathbf{K}$-invariant. 

We call a spin graph function $f\in\mathcal{H}$
{\it elementary} if $f$ is $\pi$-elementary for some
$\pi\in (G^E)^\wedge$. We denote 
\[
\mathcal{H}^\pi=\mathcal{H}^\pi_{\Gamma,G,S}
\] 
for the space $(\mathcal{R}^\pi(G^E)\otimes S)^{\mathbf{K}}$ of $\pi$-elementary spin graph functions.

Note that the elementary spin graph functions span $\mathcal{H}$, since 
\[
(\mathcal{R}(G^E)\otimes S)^{\mathbf{K}}=\bigoplus_{\pi\in (G^E)^\wedge}(\mathcal{R}^\pi(G^E)\otimes S)^{\mathbf{K}}
\]
by the Peter-Weyl theorem (see \S\ref{4-3}). 

\subsection{}\label{4-9}
Let $\pi_{e_j}=\pi_j: G\rightarrow\textup{GL}(M_j)$ be finite dimensional $G$-representations, attached to the edges of $\Gamma$. The associated tensor product representation $\bm{\pi}: G^E\rightarrow\textup{GL}(\mathbf{M})$ is defined by
\[
\bm{\pi}(\bm{g}):=\pi_1(g_1)\otimes\cdots\otimes\pi_n(g_n), 
\]
where $\mathbf{M}:=M_1\otimes\cdots\otimes M_n$.
It is convenient to think of the local representations $\pi_e$ ($e\in E$) as a choice of coloring of the
edges of $\Gamma$. 

We will identify $\mathbf{M}^*\simeq M_1^*\otimes\cdots\otimes M_n^*$ as in \S\ref{4-9}. In particular, for pure tensors $\bm{m}:=m_1\otimes\cdots\otimes m_n\in\mathbf{M}$
and $\bm{\phi}:=\phi_1\otimes\cdots\otimes \phi_n\in M_1^*\otimes\cdots\otimes M_n^*$, the matrix coefficient $c_{\bm{\phi},\bm{m}}^{\bm{\pi}}$ of $\mathbf{M}$ is 
\[
c_{\bm{\phi},\bm{m}}^{\bm{\pi}}(\bm{g})=c_{\phi_1,m_1}^{\pi_1}(g_1)\cdots c_{\phi_n,m_n}^{\pi_n}(g_n),\qquad\quad\bm{g}\in\mathbf{G}.
\]

When the $\pi_j: G\rightarrow\textup{GL}(M^{\pi_j})$ are all irreducible, then
$\bm{\pi}$ is irreducible and its representation space will be denoted by $\mathbf{M}^{\bm{\pi}}$. 
The assigment $(\pi_e)_{e\in E}\mapsto \bm{\pi}$
induces a bijection 
\[
(G^\wedge)^{\times n}\overset{\sim}{\longrightarrow} (G^E)^\wedge.
\]
If $\pi\in (G^E)^\wedge$ then we call the $\pi_{e}\in G^\wedge$ such that
$\pi\simeq \bm{\pi}$ the {\it local components of $\pi$}. 
The local components of $\pi^*$ are $\pi_e^*$.

\subsection{}\label{4-10}
Similarly we denote tensor product representations of the gauge group $\mathbf{K}=\prod_{v\in V}K_v$ by  
 $\bm{\sigma}: \mathbf{K}\rightarrow\textup{GL}(\mathbf{S})$ with
\[
\bm{\sigma}(\bm{k}):=\sigma_1(k_1)\otimes\cdots\otimes\sigma_r(k_r)
\]
and $\mathbf{S}:=S_1\otimes\cdots\otimes S_r$, where $\sigma_{v_i}=\sigma_i: K_{v_i}\rightarrow\textup{GL}(S_i)$ are finite dimensional representations of the local gauge groups $K_{v_i}$. We now think of the local representations $\sigma_v$ ($v\in V$) as a choice of coloring of the vertices of $\Gamma$. 

If $\sigma: \mathbf{K}\rightarrow\textup{GL}(S)$ is a finite dimensional irreducible representation then $\sigma$ is isomorphic to a tensor product representation with
finite dimensional irreducible local $K_v$-repre\-sen\-tations $\sigma_v: K_v\rightarrow\textup{GL}(S_v)$.

\subsection{}\label{4-11}
The product group $G^E\times G^E$ acts on $\mathcal{R}(G)^{\otimes n}$ by 
\[
(\bm{g}^\prime,\bm{g}^{\prime\prime})\cdot (f_1\otimes\cdots\otimes f_n):=(g_1^\prime,g_1^{\prime\prime})\cdot f_1\otimes\cdots\otimes
(g_n^\prime,g_n^{\prime\prime})\cdot f_n,
\]
where $\bm{g}^\prime=(g_1^\prime,\ldots,g_n^\prime)\in G^E$ and $\bm{g}^{\prime\prime}=(g_1^{\prime\prime},\ldots,g_n^{\prime\prime})\in G^E$.
The linear isomorphism
\begin{equation}\label{isoFUN}
\begin{split}
&\mathcal{R}(G)^{\otimes n}\overset{\sim}{\longrightarrow}\mathcal{R}(G^E)\\
&f_1\otimes\cdots\otimes f_n\mapsto\bm{f},
\end{split}
\end{equation}
with $\bm{f}\in\mathcal{R}(G^E)$ defined by 
\[
\bm{f}(\bm{g}):=f_1(g_1)\cdots f_n(g_n),
\] 
intertwines the $G^E\times G^E$-actions.

\subsection{}\label{4-12}
For $\pi_j\in G^\wedge$ ($1\leq j\leq n$) the isomorphism \eqref{isoFUN} restricts to an isomorphism
\begin{equation*}
\mathcal{R}^{\pi_1}(G)\otimes\cdots\otimes\mathcal{R}^{\pi_n}(G)\overset{\sim}{\longrightarrow} \mathcal{R}^{\bm{\pi}}(G^E).
\end{equation*}
It maps $c_{\phi_1,m_1}^{\pi_1}\otimes\cdots\otimes c_{\phi_n,m_n}^{\pi_n}$ to 
$c_{\bm{\phi},\bm{m}}^{\bm{\pi}}$,
where $\bm{\phi}:=\phi_1\otimes\cdots\otimes\phi_n$ and $\bm{m}:=m_1\otimes\cdots\otimes m_n$, cf. \S\ref{4-9}.

\subsection{}\label{4-13}

The {\it star} $\mathcal{S}(v)$ of $v\in V$ is the set of edges $e\in E$ with source and/or target equal to $v$.
Then
\[
\mathcal{S}(v)=\mathcal{S}(v|s)\cup\mathcal{S}(v|t)
\]
where $\mathcal{S}(v|s)$ is the set of edges oriented outward of $v$ and 
$\mathcal{S}(v|t)$ the set of edges oriented toward $v$. Note that the union may not be disjoint since we allow loops in the graph $\Gamma$.

We consider the $K_v$-representation
\[
\pi_{\mathcal{S}(v)}: K_v\rightarrow \textup{GL}(\mathbf{M}^{\pi_{\mathcal{S}(v)}})
\]
with 
\[
\mathbf{M}^{\pi_{\mathcal{S}(v)}}:=\Bigl(\bigotimes_{e\in\mathcal{S}(v|s)}M^{\pi_e^*}\Bigr)\otimes\Bigl(\bigotimes_{e\in\mathcal{S}(v|t)}M^{\pi_e}\Bigr)
\]
and $K_v$ acting diagonally,
\[
\pi_{\mathcal{S}(v)}(k_v):=\Bigl(\bigotimes_{e\in\mathcal{S}(v|s)}\pi_e^*(k_v)\Bigr)\otimes\Bigl(\bigotimes_{e\in\mathcal{S}(v|t)}\pi_e(k_v)\Bigr).
\]
Here the tensor factors are ordered using the total order on $\mathcal{S}(v|s)$ and $\mathcal{S}(v|t)$ induced from the total order on $E$. 

We will identify the dual $K_v$-representation $\bigl(\mathbf{M}^{\pi_{\mathcal{S}(v)}}\bigr)^*$ with 
\[
\Bigl(\bigotimes_{e\in\mathcal{S}(v|s)}M^{\pi_e}\Bigr)\otimes\Bigl(\bigotimes_{e\in\mathcal{S}(v|t)}M^{\pi_e^*}\Bigr),
\]
with $K_v$ acting diagonally, cf. \S\ref{4-4}.
\subsection{}\label{4-14}

Fix finite dimensional $K_v$-representations $\sigma_v: K_v\rightarrow\textup{GL}(S_v)$ for $v\in V$ and finite dimensional $G$-representations $\pi_e: G\rightarrow\textup{GL}(M_e)$ for $e\in E$, thus providing the vertices and the edges of $\Gamma$ with representation colors.

At vertex $v\in V$ we assign to the colored graph $\Gamma$ the $K_v$-representation
\[
\mathbf{M}^{\pi_{\mathcal{S}(v)}}\otimes S_v,
\]
with $K_v$ acting diagonally, and we endow 
\[
\bigotimes_{v\in V}\bigl(\mathbf{M}^{\pi_{\mathcal{S}(v)}}\otimes S_v\bigr)
\]
with the tensor product action of the gauge group $\mathbf{K}=\prod_{v\in V}K_v$. Concretely, 
\[
\bm{k}\cdot\bigotimes_{v\in V}\bigl(C_v\otimes u_v\bigr):=\bigotimes_{v\in V}\bigl(\pi_{\mathcal{S}(v)}(k_v)C_v\otimes\sigma_v(k_v)u_v\bigr)
\]
for $\bm{k}=(k_v)_{v\in V}\in\mathbf{K}$, $C_v\in \mathbf{M}^{\pi_{\mathcal{S}(v)}}$ and $u_v\in S_v$.

\subsection{}\label{4-15}
For $\pi_e\in G^\wedge$ and $\sigma_v: K_v\rightarrow\textup{GL}(S_v)$ finite dimensional $K_v$-representations,
consider the linear map
\[
\Psi^{\bm{\pi}}:\, \bigotimes_{v\in V}\bigl(\mathbf{M}^{\pi_{\mathcal{S}(v)}}\otimes S_v\bigr)\rightarrow\mathcal{R}^{\bm{\pi}}(G^E)\otimes\mathbf{S}
\]
defined on pure tensors by
\begin{equation}\label{Psipure}
\Psi^{\bm{\pi}}\Bigl(\bigotimes_{v\in V}\Bigl(\Bigl(\bigotimes_{e\in\mathcal{S}(v|s)}\phi_e\Bigr)\otimes\Bigl(\bigotimes_{e\in\mathcal{S}(v|t)}m_e\Bigr)\otimes u_v\Bigr)\Bigr):=
c_{\bm{\phi},\bm{m}}^{\bm{\pi}}\otimes\bm{u}
\end{equation}
with $\bm{\phi}:=\bigotimes_{e\in E}\phi_e$, $\bm{m}:=\bigotimes_{e\in E}m_e$ and $\bm{u}:=\bigotimes_{v\in V}u_v$. This is well defined due to the disjoint union decompositions
\[
\bigsqcup_{v\in V}\mathcal{S}(v|s)=E=\bigsqcup_{v\in V}\mathcal{S}(v|t)
\]
of the edge set $E$ of $\Gamma$.

\begin{thm}
The linear map $\Psi^{\bm{\pi}}$ defines a $\mathbf{K}$-linear isomorphism
\[
\Psi^{\bm{\pi}}:\, \bigotimes_{v\in V}\bigl(\mathbf{M}^{\pi_{\mathcal{S}(v)}}\otimes S_v\bigr)\overset{\sim}{\longrightarrow}
\mathcal{R}^{\bm{\pi}}(G^E)\otimes\mathbf{S}.
\]
It restricts to a linear isomorphism
\[
\Psi^{\bm{\pi}}:\, \bigotimes_{v\in V}\bigl(\mathbf{M}^{\pi_{\mathcal{S}(v)}}\otimes S_v\bigr)^{K_v}\overset{\sim}{\longrightarrow}
\bigl(\mathcal{R}^{\bm{\pi}}(G^E)\otimes\mathbf{S}\bigr)^{\mathbf{K}}.
\]
\end{thm}
\begin{proof}
By \S\ref{4-12} it is clear that $\Psi^{\bm{\pi}}$ is a linear isomorphism. It is $\mathbf{K}$-linear, since for $\bm{k}=(k_v)_{v\in V}\in\mathbf{K}$ and $\bm{g}\in G^E$,
\begin{equation*}
\begin{split}
\Psi^{\bm{\pi}}\Bigl(\bigotimes_{v\in V}&\Bigl(\Bigl(\bigotimes_{e\in\mathcal{S}(v|s)}\phi_e\pi_e(k_v^{-1})\Bigr)\otimes\Bigl(\bigotimes_{e\in\mathcal{S}(v|t)}\pi_e(k_v)m_e\Bigr)\otimes 
\sigma_v(k_v)u_v\Bigr)\Bigr)(\bm{g})=\\
&\qquad\qquad\quad=\Bigl(\prod_{e\in E}\phi_e\bigl(\pi_e(k_{s(e)}^{-1}g_ek_{t(e)})m_e\bigr)\Bigr)\bm{\sigma}(\bm{k})\bm{u}\\
&\qquad\qquad\quad=
c_{\bm{\phi},\bm{m}}^{\bm{\pi}}(\bm{k}^{-1}\cdot\bm{g})\bm{\sigma}(\bm{k})\bm{u}.
\end{split}
\end{equation*}
The second statement of the theorem follows now immediately from the fact that $\bigotimes_{v\in V}\bigl(\mathbf{M}^{\pi_{\mathcal{S}(v)}}\otimes S_v\bigr)$ is endowed with 
the tensor product action of $\mathbf{K}=\prod_{v\in V}K_v$, see \S\ref{4-14}.
\end{proof}

\subsection{}\label{4-16}
We keep the setup as in \S\ref{4-15}.
For $e\in E$ let $\{m_{e,i_e}\}_{i_e\in\mathcal{I}_e}$ be a basis of $M^{\pi_e}$. For
$v\in V$ set
\[
\mathcal{I}(v\vert s):=\{\bm{i}=(i_e)_{e\in\mathcal{S}(v|s)} \,\, | \,\, i_e\in \mathcal{I}_e\}
\]
and write for $\bm{i}\in \mathcal{I}(v|s)$,
\begin{equation*}
\bm{m}_{\bm{i}}(v|s):=\bigotimes_{e\in\mathcal{S}(v|s)}m_{e,i_e}.
\end{equation*}
In a similar way we define $\bm{m}_{\bm{i}}(v|t)$ for indices $\bm{i}$ from 
\[
\mathcal{I}(v|t):=\{\bm{i}=(i_e)_{e\in\mathcal{S}(v|t)}\,\, | \,\, i_e\in\mathcal{I}_e\}.
\]
Then 
\begin{equation}\label{basis}
\{\,\bm{m}_{\bm{i}}(v|s)^*\otimes\bm{m}_{\bm{j}}(v|t)\,\,\,\,\,| \,\,\,\,\, (\bm{i},\bm{j})\in\mathcal{I}(v|s)\times \mathcal{I}(v|t)\,\}
\end{equation}
is a basis of $\mathbf{M}^{\pi_{\mathcal{S}(v)}}$.
We then have a linear isomorphism 
\begin{equation}\label{Kiso}
\textup{Hom}_{K_v}\bigl((\mathbf{M}^{\pi_{\mathcal{S}(v)}})^*,S_v\bigr)
\overset{\sim}{\longrightarrow}
\bigl(\mathbf{M}^{\pi_{\mathcal{S}(v)}}\otimes S_v\bigr)^{K_v}
\end{equation}
mapping $\Phi\in\textup{Hom}_{K_v}\bigl((\mathbf{M}^{\pi_{\mathcal{S}(v)}})^*,S_v\bigr)$ to the invariant tensor
\[
\sum_{\bm{i}\in\mathcal{I}(v|s)}\sum_{\bm{j}\in\mathcal{I}(v|t)}\bm{m}_{\bm{i}}(v|s)^*\otimes\bm{m}_{\bm{j}}(v|t)\otimes\Phi\bigl(\bm{m}_{\bm{i}}(v|s)\otimes \bm{m}_{\bm{j}}(v|t)^*\bigr).
\] 

Combined with Theorem \ref{4-15} we thus obtain a parametrisation of the space
$\bigl(\mathcal{R}^{\bm{\pi}}(G^E)\otimes\mathbf{S}\bigr)^{\mathbf{K}}$ of ${\bm{\pi}}$-elementary spin graph functions in terms of spaces of local intertwiners (local in the sense that they only depend on the colors of $\Gamma$ at the star of a vertex $v$).

\subsection{}\label{4-17}
Let $n\geq 1$ and consider the directed cycle graph $\Gamma$ with $n$ edges.
We enumerate the vertices $v_i$ and edges $e_j$ ($i,j\in\mathbb{Z}/n\mathbb{Z}$) in such a way that $s(e_i)=v_{i}$ and $t(e_i)=v_{i+1}$ for $i\in\mathbb{Z}/n\mathbb{Z}$.
The order $1<2<\cdots<n$ on $\mathbb{Z}/n\mathbb{Z}$ provides a total order on $V$ and $E$.
We take $\mathbf{K}=G^V$ as gauge group.

With these conventions $G^E\simeq G^{\times n}$ by $(g_e)_{e\in E}\mapsto (g_{e_1},\ldots,g_{e_n})$, and
  $\mathbf{K}\simeq G^{\times n}$ by
 $(k_v)_{v\in V}\mapsto (k_{v_1},\ldots,k_{v_n})$.
The left gauge action of $\mathbf{K}$ on $G^E$ then reads as
\[
\bm{k}\cdot \bm{g}:=(k_1g_1k_2^{-1},k_2g_2k_3^{-1},\ldots, k_ng_nk_1^{-1})
\]
for $\bm{k}=(k_1,\ldots,k_n)\in G^{\times n}\simeq\mathbf{K}$ and $\bm{g}=(g_1,\ldots,g_n)\in G^{\times n}\simeq G^E$.

Let $\sigma_i: G\rightarrow\textup{GL}(S_i)$ be finite dimensional $K_{v_i}=G$-representations attached to the vertices $v_i$,
and $\pi_i: G\rightarrow\textup{GL}(M^{\pi_i})$ finite dimensional irreducible $G$-representation attached to the edge $e_i$. 

The partial trace $\textup{Tr}_{M^{\pi_n}}^{\mathbf{S}}(B)$ of $B\in \textup{Hom}(M^{\pi_n},M^{\pi_n}\otimes\mathbf{S})$ is the unique vector in $\mathbf{S}$ satisfying
\[
\phi\bigl(\textup{Tr}_{M^{\pi_n}}^{\mathbf{S}}(B)\bigr)
=\textup{Tr}_{M^{\pi_n}}\bigl((\textup{id}_{M^{\pi_n}}\otimes\phi)B\bigr)\qquad\quad \forall\, \phi\in\mathbf{S}^*,
\]
where $\textup{Tr}_{M^{\pi_n}}$ is the usual trace on $\textup{End}(M^{\pi_n})$. The elementary spin graph functions now have the following description in terms of partial traces of compositions of intertwiners.

\begin{prop}
Endow $M^{\pi_{i-1}}\otimes S_i$
with the diagonal $G$-action. We have a linear isomorphism
\begin{equation*}
\begin{split}
\bigotimes_{i\in\mathbb{Z}/n\mathbb{Z}}\textup{Hom}_G\bigl(M^{\pi_i},M^{\pi_{i-1}}\otimes\mathbf{S}_i\bigr)&\overset{\sim}{\longrightarrow}
\bigl(\mathcal{R}^{\bm{\pi}}(G^{\times n})\otimes\mathbf{S}\bigr)^{\mathbf{K}},\\
\bigotimes_{i\in\mathbb{Z}/n\mathbb{Z}}\Phi_i&\mapsto f_{\bm{\Phi}}^{\bm{\pi}}
\end{split}
\end{equation*}
with $\bm{\Phi}:=(\Phi_1,\ldots,\Phi_n)$ and $f_{\bm{\Phi}}^{\bm{\pi}}\in\bigl(\mathcal{R}^\pi(G^{\times n})\otimes\mathbf{S}\bigr)^{\mathbf{K}}$ the {\it elementary $n$-point trace function}
\[
f_{\bm{\Phi}}^{\bm{\pi}}(\bm{g}):=\textup{Tr}_{M^{\pi_n}}^{\mathbf{S}}\Bigl((\Phi_1\pi_1(g_1)\otimes\textup{id}_{\mathbf{S}_2\otimes\cdots\otimes\mathbf{S}_n})\cdots
(\Phi_{n-1}\pi_{n-1}(g_{n-1})\otimes\textup{id}_{\mathbf{S}_n})\Phi_n\pi_n(g_n)\Bigr).
\]
\end{prop}
\begin{proof}
In the current situation we have
$\mathbf{M}^{\pi_{\mathcal{S}(v_i)}}=
M^{\pi_i^*}\otimes M^{\pi_{i-1}}$. Using \eqref{Kiso} and the fact that
$K_{v_i}=G$ we obtain
\begin{equation*}
\begin{split}
\bigl(\mathbf{M}^{\pi_{\mathcal{S}(v_i)}}\otimes S_i\bigr)^{K_{v_i}}&\simeq
\textup{Hom}_{G}\bigl(M^{\pi_i}\otimes M^{\pi_{i-1}^*},S_i)\\
&\simeq
\textup{Hom}_{G}(M^{\pi_i},M^{\pi_{i-1}}\otimes S_i),
\end{split}
\end{equation*}
where
$M^{\pi_i}\otimes M^{\pi_{i-1}^*}$ is considered as $G$-representations via the diagonal $G$-action. The isomorphism $\textup{Hom}_{G}(M^{\pi_i},M^{\pi_{i-1}}\otimes S_i)\overset{\sim}{\longrightarrow}
\textup{Hom}_G(M^{\pi_i}\otimes M^{\pi_{i-1}^*},S_i)$ maps $\Phi_i$ to the $G$-intertwiner $m_i\otimes\phi_{i-1}\mapsto (\phi_{i-1}\otimes\textup{id}_{S_i})\Phi_i(m_i)$.

Under these identifications the intertwiner 
$\Phi_i\in \textup{Hom}_G\bigl(M^{\pi_i},M^{\pi_{i-1}}\otimes S_i\bigr)$ 
corresponds to the local invariant tensor 
\[
\widetilde{\Phi}_i:=\sum_{k\in\mathcal{I}_{e_i}}m_{e_i,k}^{*}\otimes\Phi_i(m_{e_i,k})\in \bigl(\mathbf{M}^{\pi_{\mathcal{S}(v_i)}}\otimes S_i\bigr)^{K_{v_i}}.
\]
Combined with Theorem \ref{4-15} we thus obtain 
a linear isomorphism 
\[
\bigotimes_{i\in\mathbb{Z}/n\mathbb{Z}}\textup{Hom}_G\bigl(M^{\pi_i},M^{\pi_{i-1}}\otimes S_i\bigr)\overset{\sim}{\longrightarrow}
\bigl(\mathcal{R}^{\bm{\pi}}(G^{\times n})\otimes\mathbf{S}\bigr)^{\mathbf{K}},\qquad
\bigotimes_{i\in\mathbb{Z}/n\mathbb{Z}}\Phi_i\mapsto\widetilde{f}_{\bm{\Phi}}^{\bm{\pi}}
\]
with $\widetilde{f}_{\bm{\Phi}}^{\bm{\pi}}:=\Psi^{\bm{\pi}}\bigl(\bigotimes_{i\in\mathbb{Z}/n\mathbb{Z}}\widetilde{\Phi}_i\bigr)$. 

Rewriting $\widetilde{\Phi}_i$ as 
\begin{equation}\label{tildePhi}
\widetilde{\Phi}_i=\sum_{k_{i-1}\in\mathcal{I}_{e_{i-1}}}\sum_{\ell_i\in\mathcal{I}_{e_{i}}}m_{e_i,\ell_i}^*\otimes m_{e_{i-1},k_{i-1}}\otimes \bigl((m_{e_{i-1},k_{i-1}}^*\otimes\textup{id}_{S_i})\Phi_i(m_{e_i,\ell_i})\bigr)
\end{equation}
and applying \eqref{Psipure}, we obtain the explicit expression
\[
\widetilde{f}_{\bm{\Phi}}^{\bm{\pi}}(\bm{g})=\sum_{k_1\in\mathcal{I}_{e_1}}\cdots\sum_{k_n\in\mathcal{I}_{e_n}}\bigotimes_{i\in\mathbb{Z}/n\mathbb{Z}}\bigl(m_{e_{i-1},k_{i-1}}^{*}\pi_{i-1}(g_{i-1})\otimes\textup{id}_{\mathbf{S}_i}\bigr)\Phi_i(m_{e_i,k_i}).
\]
For $i\not=n$ the sum over $k_i\in\mathcal{I}_{e_i}$ can be simplified using the identity
\begin{equation*}
\begin{split}
\sum_{k_i\in\mathcal{I}_{e_i}}\Phi_i(m_{e_i,k_i})\otimes \bigl((m_{e_i,k_i}^{*}\pi_i(g_i)\otimes\textup{id}_{S_{i+1}})&\Phi_{i+1}(m_{e_{i+1},k_{i+1}})\bigr)=\\
&=(\Phi_i\pi_i(g_i)\otimes\textup{id}_{S_{i+1}})\Phi_{i+1}(m_{e_{i+1},k_{i+1}})
\end{split}
\end{equation*}
in $M^{\pi_{i-1}}\otimes S_i\otimes S_{i+1}$.
The expression for $\widetilde{f}_{\bm{\Phi}}^{\bm{\pi}}(\bm{g})$ then reduces to
\begin{equation*}
\begin{split}
\widetilde{f}_{\bm{\Phi}}^{\bm{\pi}}&(\bm{g})=\sum_{k_n\in\mathcal{I}_{e_n}}\left\{\bigl(m_{e_n,k_n}^{*}\pi_n(g_n)\otimes\textup{id}_{\mathbf{S}}\bigr)\right.
(\Phi_1\pi_1(g_1)\otimes\textup{id}_{S_2\otimes\cdots\otimes S_n})\cdots\\
&\left.\qquad\quad\qquad\qquad\qquad\qquad\qquad\qquad\cdots(\Phi_{n-1}\pi_{n-1}(g_{n-1})\otimes\textup{id}_{S_n})\Phi_n(m_{e_n,k_n})\right\}\\
=&\textup{Tr}_{M^{\pi_n}}^{\bm{S}}\Bigl((\pi_n(g_n)\otimes\textup{id}_{\mathbf{S}})(\Phi_1\pi_1(g_1)\otimes\textup{id}_{S_2\otimes\cdots\otimes S_n})\cdots
(\Phi_{n-1}\pi_{n-1}(g_{n-1})\otimes\textup{id}_{S_n})\Phi_n\Bigr).
\end{split}
\end{equation*}
Hence $\widetilde{f}_{\bm{\Phi}}^{\bm{\pi}}=f_{\bm{\Phi}}^{\bm{\pi}}$ by the cyclicity of the partial trace:
\[
\textup{Tr}_{M^{\pi_n}}^{\mathbf{S}}((A\otimes\textup{id}_{\mathbf{S}})B)=\textup{Tr}_{M^{\pi_n}}^{\mathbf{S}}(BA)
\]
for $A\in\textup{End}(M^{\pi_n})$ and $B\in\textup{End}(M^{\pi_n}\otimes\mathbf{S})$.
\end{proof}
The study of {\it $n$-point trace functions} originates from the paper \cite{ES}. Intertwiners
$\textup{Hom}_G\bigl(M^{\pi_i},M^{\pi_{i-1}}\otimes S_i\bigr)$ may be viewed as a topological degenerations of vertex operators. The class of (elementary) $n$-point trace functions and its generalisation to the affine and quantum group level
are particularly well studied, see, e.g., \cite{ES,EV,ESV}.

\subsection{}\label{4-18}
Let $n\in\mathbb{Z}_{\geq 1}$. As a next example we consider the linearly ordered linear graph $\Gamma$ with $n$ edges. We denote the ordered vertex and edge sets by
$V=\{v_1,\ldots,v_{n+1}\}$ and $E=\{e_1,\ldots,e_{n}\}$. The source and target maps are $s(e_i)=v_i$ and $t(e_i)=v_{i+1}$ for $i=1,\ldots,n$. As local gauge groups we take
\begin{equation}\label{localalg}
K_{v_i}=
\begin{cases}
H\qquad &\hbox{ for }\, i=1,\\
G\qquad &\hbox{ for }\, i=2,\ldots,n,\\
K\qquad &\hbox{ for }\, i=n+1,
\end{cases}
\end{equation}
where $H,K\subseteq G$ are subgroups. 

The action of the associated gauge group 
\[
\mathbf{K}=\prod_{i=1}^{n+1}K_{v_i}=H\times G^{\times (n-1)}\times K
\]
on $G^E\simeq G^{\times n}$ then becomes
\[
\bm{k}\cdot\bm{g}^\prime:=(hg_1^\prime g_2^{-1},g_2g_2^\prime g_3^{-1},\ldots,g_ng_n^\prime k^{-1})
\]
for $\bm{k}=(h,g_2,\ldots,g_{n},k)\in \mathbf{K}$ and $\bm{g}^\prime=(g_1^\prime,\ldots,g_n^\prime)\in G^E$.

Let $S_i$ ($1<i\leq n$) be finite dimensional $G$-representations, $L$ a finite dimensional $H$-representation and $N$ a finite dimensional  $K$-representation. Denote by
\[
\mathbf{S}:=L\otimes S_2\otimes\cdots\otimes S_n\otimes N
\]
the resulting tensor product spin representation of $\mathbf{K}$. We also write 
\[
\underline{S}:=S_2\otimes\cdots\otimes S_n,
\]
for the ``bulk'' $G^{\times (n-1)}$-representation associated to $\mathbf{S}$, so that $\mathbf{S}=L\otimes\underline{S}\otimes N$.
Denote by 
\[
Q_{L\otimes\underline{S},N}: \textup{Hom}(N^*,L\otimes\underline{S})\overset{\sim}{\longrightarrow}\mathbf{S}
\]
the linear isomorphism defined by 
\[
Q_{L\otimes\underline{S},N}(T):=\sum_{j}T(n_j^{*})\otimes n_j,
\]
with $\{n_j\}_j$ a basis of $N$ and $\{n_j^*\}_j$ the corresponding dual basis of $N^*$.

Let $\bm{\pi}: G^{\times n}\rightarrow\textup{GL}(\mathbf{M}^{\bm{\pi}})$ be an irreducible tensor product representation, with local components $\pi_i: G\rightarrow\textup{GL}(M^{\pi_i})$ ($1\leq i\leq n$). In the following proposition we will also view $M^{\pi_1}$ (resp. $M^{\pi_n}$) as $H$-representation (resp. $K$-representation) by restriction.
\begin{prop}
We have a linear isomorphism
\begin{equation*}
\begin{split}
\textup{Hom}_{H}(M^{\pi_1},L)\otimes\,\bigotimes_{i=2}^n\textup{Hom}_G\bigl(M^{\pi_i},M^{\pi_{i-1}}\otimes S_i\bigr)&\otimes
\textup{Hom}_{K}(N^*,M^{\pi_n})\\
&\qquad\overset{\sim}{\longrightarrow}
\bigl(\mathcal{R}^{\bm{\pi}}(G^{\times n})\otimes\mathbf{S}\bigr)^{\mathbf{K}}
\end{split}
\end{equation*}
mapping $\Theta\otimes\bigl(\bigotimes_{i=2}^n\Phi_i\bigr)\otimes\Xi$ to the spin graph function $f_{\Theta,\bm{\Phi},\Xi}^{\bm{\pi}}\in\bigl(\mathcal{R}^{\bm{\pi}}(G^{\times n})\otimes\mathbf{S}\bigr)^{\mathbf{K}}$, defined by
\begin{equation*}
\begin{split}
f_{\Theta,\bm{\Phi},\Xi}^{\bm{\pi}}(\bm{g}):=&Q_{L\otimes\underline{S},N}\Bigl((\Theta\pi_1(g_1)\otimes\textup{id}_{\underline{S}})
(\Phi_2\pi_2(g_2)\otimes\textup{id}_{S_3\otimes\cdots\otimes S_{n}})
\cdots\\
&\,\,\,\quad\qquad\qquad\qquad\qquad\qquad\quad\cdots
(\Phi_{n-1}\pi_{n-1}(g_{n-1})\otimes\textup{id}_{S_n})\Phi_n\pi_n(g_n)\Xi\Bigr).
\end{split}
\end{equation*}
\end{prop}
\begin{proof}
At vertices $v_i$ with $1<i\leq n$ the analysis of the local space of invariants $\bigl(\mathbf{M}^{\pi_{\mathcal{S}(v_i)}}\otimes S_i\bigr)^{K_{v_i}}$ is as in the proof of Proposition \ref{4-17}. For $i=1$ we have
\[
\bigl(\mathbf{M}^{\pi_{\mathcal{S}(v_1)}}\otimes L\bigr)^{K_{v_1}}=\bigl(M^{\pi_1^*}\otimes L\bigr)^H\simeq\textup{Hom}_H\bigl(M^{\pi_1},L),
\]
with the isomorphism as in \S\ref{4-16}. For $i=n+1$ we analogously have
\[
\bigl(\mathbf{M}^{\pi_{\mathcal{S}(v_{n+1})}}\otimes N\bigr)^{K_{v_{n+1}}}=\bigl(M^{\pi_{n}}\otimes N\bigr)^{K}\simeq
\textup{Hom}_K(N^*,M^{\pi_{n}}).
\]
Under these isomorphisms the intertwiner $\Theta\in\textup{Hom}_H\bigl(M^{\pi_1},L)$ corresponds 
\[
\widetilde{\Theta}:=\sum_{\ell_1\in\mathcal{I}_{e_1}}m_{e_1,\ell_1}^*\otimes\Theta(m_{e_1,\ell_1})
\in \bigl(\mathbf{M}^{\pi_{\mathcal{S}(v_1)}}\otimes L\bigr)^{K_{v_1}}
\]
and $\Xi\in\textup{Hom}_K(N^*,M^{\pi_{n}})$ to 
\[
\widetilde{\Xi}:=\sum_j\Xi(n_j^*)\otimes n_j\in \bigl(\mathbf{M}^{\pi_{\mathcal{S}(v_{n+1})}}\otimes N\bigr)^{K_{v_{n+1}}}.
\]

Combined with Theorem \ref{4-15} we thus obtain 
a linear isomorphism 
\begin{equation*}
\begin{split}
\textup{Hom}_{H}(M^{\pi_1},L)\otimes\,\bigotimes_{i=2}^n\textup{Hom}_G\bigl(M^{\pi_i},M^{\pi_{i-1}}\otimes S_i\bigr)&\otimes
\textup{Hom}_{K}(N^*,M^{\pi_n})\\
&\qquad\overset{\sim}{\longrightarrow}
\bigl(\mathcal{R}^{\bm{\pi}}(G^{\times n})\otimes\mathbf{S}\bigr)^{\mathbf{K}}
\end{split}
\end{equation*}
mapping $\Theta\otimes\bigl(\bigotimes_{i=2}^n\Phi_i\bigr)\otimes\Xi$ to 
\[
\widetilde{f}_{\Theta,\bm{\Phi},\Xi}^{\bm{\pi}}:=\Psi^{\bm{\pi}}\Bigl(\widetilde{\Theta}\otimes\bigl(\bigotimes_{i=2}^n\widetilde{\Phi}_i\bigr)\otimes\widetilde{\Xi}\Bigr),
\]
where $\widetilde{\Phi}_i$ is given by
\eqref{tildePhi}.

A direct computation now shows that the spin graph function $\widetilde{f}_{\Theta,\bm{\Phi},\Xi}^{\bm{\pi}}(\bm{g})$ is explicitly given by
\begin{equation*}
\begin{split}
\sum_{i_1\in\mathcal{I}_{e_1}}\cdots\sum_{i_{n-1}\in\mathcal{I}_{e_{n-1}}}\sum_j&\Theta(m_{e_1,i_1})\otimes\bigl((m_{e_1,i_1}^{*}\pi_1(g_1)\otimes\textup{id}_{S_2})
\Phi_2(m_{e_2,i_2})\bigr)\otimes\cdots\\
&\cdots\otimes \bigl((m_{e_{n-1},i_{n-1}}^{*}\pi_{n-1}(g_{n-1})\otimes\textup{id}_{S_n})\Phi_n(\pi_n(g_n)\Xi(n_j^*))\bigr)\otimes n_j.
\end{split}
\end{equation*}
Contracting the bulk intertwiners $\Phi_i$ ($i=2,\ldots,n$) as in the proof of Proposition \ref{4-17}, we obtain the expression
\begin{equation*}
\begin{split}
\sum_{i_1\in\mathcal{I}_{e_1}}\sum_j&\Theta(m_{e_1,i_1})\otimes
\left\{(m_{e_1,i_1}^{*}\pi_1(g_1)\otimes\textup{id}_{\underline{S}})(\Phi_2\pi_2(g_2)\otimes\textup{id}_{S_3\otimes\cdots\otimes S_n})\cdots\right.\\
&\left.\qquad\qquad\qquad\qquad\quad\cdots\bigl(\Phi_{n-1}\pi_{n-1}(g_{n-1})\otimes\textup{id}_{S_n}\bigr)\Phi_n(\pi_n(g_n)\Xi(n_j^*))\right\}\otimes n_j
\end{split}
\end{equation*}
for $\widetilde{f}_{\Theta,\bm{\Phi},\Xi}^{\bm{\pi}}(\bm{g})$, which is easily seen to be equal to $f_{\Theta,\bm{\Phi},\Xi}^{\bm{\pi}}(\bm{g})$. This completes the proof.
\end{proof}
For $H=K$ the modified spin graph functions $Q_{L\otimes\underline{S},N}^{-1}f_{\Theta,\bm{\Phi},\Xi}^{\bm{\pi}}$ are the {\it elementary $n$-point spherical functions} from \cite{SR,RS}. 
For $n=1$, they reduce to the elementary spherical functions on compact symmetric spaces.

\section{Quantum spin systems on graph connections}\label{5}
Let $\Gamma$ be a finite oriented graph and $G$ a connected compact Lie group. We write 
$V=\{v_1,\ldots,v_r\}$ and $E=\{e_1,\ldots,e_n\}$ for the totally ordered vertex and edge set of $\Gamma$. We denote by $\mathfrak{g}$ the complexification of the Lie algebra $\mathfrak{g}_0$ of $G$. In this section we introduce a quantum spin system on the spaces $\mathcal{H}=\mathcal{H}_{\Gamma,G,S}$ of spin graph functions. 

\subsection{}\label{5-1}
We call a linear differential operator $D$ on $G$ algebraic if it preserves the space $\mathcal{R}(G)$ of representative functions on $G$ (see \S\ref{4-1}). 
We have an inclusion of algebras
\[
\mathcal{D}_{\textup{biinv}}(G)\subseteq\mathcal{D}_{\textup{inv}}(G)\subseteq\mathcal{D}(G)
\]
with $\mathcal{D}(G)$ the algebra of algebraic differential operators on $G$, with $\mathcal{D}_{\textup{inv}}(G)\subseteq\mathcal{D}(G)$ the subalgebra
generated by the left and right $G$-invariant differential operators on $G$, and with $\mathcal{D}_{\textup{biinv}}(G)\subseteq\mathcal{D}_{\textup{inv}}(G)$ the subalgebra
of $G$-biinvariant differential operators on $G$. 
\subsection{}\label{5-2}
We have a surjective algebra map 
\begin{equation}\label{su}
U(\mathfrak{g}^{\times 2})\twoheadrightarrow\mathcal{D}_{\textup{inv}}(G),\qquad X\mapsto D_X
\end{equation}
defined by
\begin{equation*}
\begin{split}
\bigl(D_{(x,y)}f\bigr)(g):=&\frac{d}{ds}\bigg|_{s=0}f\bigl(\exp(-sx)g\exp(sy)\bigr)\\
=&\frac{d}{ds}\bigg|_{s=0}f\bigl(\exp(-sx)g\bigr)+\frac{d}{dt}\bigg|_{t=0}f\bigl(g\exp(ty)\bigr)
\end{split}
\end{equation*}
for $(x,y)\in\mathfrak{g}_0^{\times 2}$. Identify $U(\mathfrak{g}^{\times 2})\simeq U(\mathfrak{g})\otimes
U(\mathfrak{g})$ as algebras, with the isomorphism $U(\mathfrak{g}^{\times 2})\overset{\sim}{\longrightarrow}
U(\mathfrak{g})\otimes U(\mathfrak{g})$ induced by $(x,y)\mapsto x\otimes 1+1\otimes y$ for $x,y\in\mathfrak{g}$. We then have the balancing condition
\[
D_{X\otimes ZY}=D_{X\iota(Z)\otimes Y}
\]
for $X,Y\in U(\mathfrak{g})$ and $Z\in Z(\mathfrak{g})$, where $\iota$ is the antipode of $U(\mathfrak{g})$ (i.e., $\iota$ is the unique anti-algebra automorphism of $U(\mathfrak{g})$
such that $x\mapsto -x$ for $x\in\mathfrak{g}$). Hence the algebra map \eqref{su} descends to an isomorphism of algebras
\begin{equation}\label{im}
U(\mathfrak{g})\otimes_{Z(\mathfrak{g})}U(\mathfrak{g})\overset{\sim}{\longrightarrow}\mathcal{D}_{\textup{inv}}(G)
\end{equation}
with the balanced tensor product over $Z(\mathfrak{g})$ relative to
the $\iota$-twisted right regular $Z(\mathfrak{g})$-action on $U(\mathfrak{g})$
\[
X\cdot Z:=X\iota(Z)\qquad\quad (X\in U(\mathfrak{g}),\,\, Z\in Z(\mathfrak{g}))
\]
and the left regular $Z(\mathfrak{g})$-action on $U(\mathfrak{g})$ (injectivity of the map \eqref{im} was shown in \cite{Ne}).

\subsection{}\label{5-3}
The algebra $\mathcal{D}_{\textup{biinv}}(G)$ of $G$-biinvariant differential operators on $G$ is isomorphic to $Z(\mathfrak{g})$ via the map
\[
Z(\mathfrak{g})\overset{\sim}{\longrightarrow}\mathcal{D}_{\textup{biinv}}(G),\qquad\quad Z\mapsto D_{1\otimes Z}=D_{\iota(Z)\otimes 1}.
\]
In particular, $\mathcal{D}_{\textup{biinv}}(G)$ is contained in the center of $\mathcal{D}_{\textup{inv}}(G)$.

\subsection{}\label{5-4}
Consider the algebra $\mathcal{D}(G^E)$ of algebraic differential operators on the connected compact Lie group $G^E$, and recall the gauge action of $\mathbf{K}$ on $G^E$ (see \S\ref{4-6}). The corresponding contragredient $\mathbf{K}$-action on $\mathcal{R}(G^E)$ is 
\[
(\mathbf{k}\cdot f)(\mathbf{g}):=f(\mathbf{k}^{-1}\cdot\mathbf{g})
\] 
for $\mathbf{k}\in\mathbf{K}$, $f\in\mathcal{R}(G^E)$ and $\mathbf{g}\in\mathbf{G}$ (this is the special case of the $\mathbf{K}$-action on $\mathcal{R}(G^E)\otimes S$
from \S\ref{4-7} when $S$ is the trivial $\mathbf{K}$-representation). This action induces an $\mathbf{K}$-action 
\[
\mathbf{K}\times\mathcal{D}(G^E)\rightarrow\mathcal{D}(G^E),\qquad\quad
(\mathbf{k},D)\mapsto \mathbf{k}\bullet D
\] 
on $\mathcal{D}(G^E)$ by algebra automorphisms such that
\begin{equation}\label{KcompD}
\mathbf{k}\cdot (Df)=(\mathbf{k}\bullet D)(\mathbf{k}\cdot f)
\end{equation}
for $\mathbf{k}\in\mathbf{K}$, $D\in\mathcal{D}(G^E)$ and $f\in\mathcal{R}(G^E)$. We denote by
$\mathcal{D}(G^E)^{\mathbf{K}}\subseteq\mathcal{D}(G^E)$ the subalgebra of $\mathbf{K}$-invariant differential operators on $G^E$.
\subsection{}\label{5-5}
As in \S\ref{5-2}, we identify
$U((\mathfrak{g}^E)^{\times 2})\simeq U(\mathfrak{g}^E)\otimes U(\mathfrak{g}^E)$. Furthermore, we identify 
$U(\mathfrak{g}^E)\simeq U(\mathfrak{g})^{\otimes\#E}$ as algebras, with the isomorphism induced by 
\[
(x_e)_{e\in E}\mapsto \sum_{i=1}^n1^{\otimes (i-1)}\otimes x_{e_i}\otimes 1^{\otimes (n-i)}
\]
for $(x_e)_{e\in E}\in\mathfrak{g}^E$. It restricts to an isomorphism $Z(\mathfrak{g}^E)\simeq Z(\mathfrak{g})^{\otimes\#E}$.

We will use the notation 
\[
X^{(i)}:=1^{\otimes (i-1)}\otimes X\otimes 1^{\otimes (n-i)}\in U(\mathfrak{g})^{\otimes\#E}
\]
for $X\in U(\mathfrak{g})$ and $i\in\{1,\ldots,n\}$, and a pure tensor in $U((\mathfrak{g}^E)^{\times 2})$ will be denoted by
$\mathbf{X}\otimes\mathbf{Y}$ with 
\[
\mathbf{X}=\bigotimes_{e\in E}X_e,\qquad\quad\mathbf{Y}=\bigotimes_{e^\prime\in E}Y_{e^\prime}.
\]
and $X_e, Y_{e^\prime}\in U(\mathfrak{g})$.
\begin{lem}
The formula
\begin{equation}\label{bulletaction}
\bm{k}\bullet\bigl(\mathbf{X}\otimes\mathbf{Y}\bigr):=\Bigl(\bigotimes_{e\in E}\textup{Ad}(k_{s(e)})X_e\Bigr)\otimes
\Bigl(\bigotimes_{e^\prime\in E}\textup{Ad}(k_{t(e^\prime)})Y_{e^\prime}\Bigr)
\end{equation}
defines an action of $\mathbf{K}$ on $U((\mathfrak{g}^E)^{\times 2})$ by algebra automorphisms. Furthermore,
\begin{equation}\label{kcd}
\bm{k}\bullet D_{\mathbf{X}\otimes\mathbf{Y}}=D_{\bm{k}\bullet (\mathbf{X}\otimes\mathbf{Y})}.
\end{equation}
\end{lem}
\begin{proof}
The first statement is immediate. For the second statement, it suffices to check \eqref{kcd}
when $\mathbf{X}=x^{(i)}$ and $\mathbf{Y}=1_{U(\mathfrak{g}^E)}$ and when $\mathbf{X}=1_{U(\mathfrak{g}^E)}$ and $\mathbf{Y}=y^{(i)}$, where $x,y\in\mathfrak{g}_0$.
When $\mathbf{X}=x^{(i)}$ and $\mathbf{Y}=1_{U(\mathfrak{g}^E)}$ we have
\begin{equation*}
\begin{split}
\bigl(\bm{k}\cdot(D_{\mathbf{X}\otimes\mathbf{Y}}f)\bigr)(\bm{g})&=\frac{d}{dt}\bigg|_{t=0}f(\cdots,\exp(-tx)k_{s(e_i)}^{-1}g_ik_{t(e_i)},\cdots)\\
&=\frac{d}{dt}\bigg|_{t=0}f(\cdots,k_{s(e_i)}^{-1}\exp(-t\textup{Ad}(k_{s(e_i)})x)g_ik_{t(e_i)},\cdots)\\
&=\bigl(D_{\mathbf{k}\bullet(\mathbf{X}\otimes\mathbf{Y})}(\bm{k}\cdot f)\bigr)(\bm{g}),
\end{split}
\end{equation*}
as desired. A similar computation proves \eqref{kcd} when $\mathbf{X}=1_{U(\mathfrak{g}^E)}$ and $\mathbf{Y}=y^{(i)}$.
\end{proof}

\subsection{}\label{5-6}
Lemma \ref{5-5} shows that $\mathcal{D}_{\textup{inv}}(G^E)$ is a $\mathbf{K}$-invariant subalgebra of $\mathcal{D}(G^E)$.
Denote by 
\[
\mathcal{D}_{\textup{inv}}(G^E)^{\mathbf{K}}\subseteq\mathcal{D}_{\textup{inv}}(G^E)
\]
the subalgebra of $\mathbf{K}$-invariant differential operators in $\mathcal{D}_{\textup{inv}}(G^E)$.

By \S\ref{5-3} and Lemma \ref{5-5} we then have the inclusion
\begin{equation}\label{inclusion}
\mathcal{D}_{\textup{biinv}}(G^E)\subseteq\mathcal{D}_{\textup{inv}}(G^E)^{\mathbf{K}}\subseteq\mathcal{D}(G^E)^{\mathbf{K}}
\end{equation}
of algebras.

\subsection{}\label{5-7}
Let $S$ be a finite dimensional $\mathbf{K}$-representation. The space $\mathcal{R}(G^E)\otimes S$ of $S$-valued representative functions on $G^E$
becomes a $\mathcal{D}(G^E)$-module by
\[
D(h\otimes u):=D(h)\otimes u
\]
for $D\in\mathcal{D}(G^E)$, $h\in\mathcal{R}(G^E)$ and $u\in S$. In addition we have the restricted gauge group $\mathbf{K}$ acts on $\mathcal{R}(G^E)\otimes S$ by the twisted $\mathbf{K}$-action $(\bm{k},f)\mapsto\bm{k}\cdot f$ from \S\ref{4-7}. Then formula \eqref{KcompD} remains true in this more general context,
\[
\mathbf{k}\cdot (Df)=(\mathbf{k}\bullet D)(\mathbf{k}\cdot f)
\]
for $\mathbf{k}\in\mathbf{K}$, $D\in\mathcal{D}(G^E)$ and $f\in\mathcal{R}(G^E)\otimes S$.

\subsection{}\label{5-8}
As a consequence of \S\ref{5-7}, the algebra $\mathcal{D}(G^E)^{\mathbf{K}}$ of $\mathbf{K}$-invariant algebraic differential
operators on $G^E$ acts on the space $\mathcal{H}_{\Gamma,G,S}=(\mathcal{R}(G^E)\otimes S)^{\mathbf{K}}$
of spin graph functions. The resulting homomorphic image of the inclusions \eqref{inclusion} of algebras in $\textup{End}(\mathcal{H}_{\Gamma,G,S})$ gives rise to the
inclusion 
\[
I_{\Gamma,G,S}\subseteq J_{\Gamma,G,S}\subseteq A_{\Gamma,G,S}
\]
of subalgebras of $\textup{End}(\mathcal{H}_{\Gamma,G,S})$.
We omit the labels $\Gamma,G,S$ if they are clear from context. Note that $I$ is contained in the center of $J$, in view of \S\ref{5-3}.

\subsection{}\label{5-9}
Following \S\ref{2-7} we view the inclusion of algebras
\[
I\subseteq J\subseteq C_A(I)\subseteq A\subseteq\textup{End}(\mathcal{H})
\]
as a quantum spin system with quantum state space $\mathcal{H}$, algebra of quantum observables $A$, algebra of quantum integrals $J$, and commutative algebra of quantum Hamiltonians $I$. For $i\in\{1,\ldots,n\}$ and $\Omega\in Z(\mathfrak{g})$ the quadratic Casimir element, the action of
\[
\Omega^{(i)}\in Z(\mathfrak{g}^E)\simeq\mathcal{D}_{\textup{biinv}}(G^E)
\] 
on $\mathcal{H}$ is a quantum Hamiltonian $H_i\in I$ of the quantum spin system. 

We call $H_i$ ($i\in\{1,\ldots,n\}$) the {\it edge-component quadratic Hamiltonians} of the quantum spin system.

\subsection{}\label{5-10}
Write $\chi_\pi: Z(\mathfrak{g})\rightarrow\mathbb{C}$ for the central character of $\pi\in G^\wedge$.
Then
\[
(G^\wedge)^{E}\hookrightarrow I^\wedge,\qquad\quad \bm{\pi}\mapsto\bm{\chi}_{\bm{\pi}}
\]
with $\bm{\chi}_{\bm{\pi}}\in I^\wedge$ determined by the formula 
\[
\bm{\chi}_{\bm{\pi}}(Z_i)=\chi_{\pi_i}(Z)
\]
for $i=1,\ldots,n$ and $Z\in Z(\mathfrak{g})$
(here the $\pi_j$ are the local components of the tensor product representation $\bm{\pi}$, see \S\ref{4-9}). It follows from \S\ref{4-12} that the space $\mathcal{H}^{\bm{\pi}}$ of $\bm{\pi}$-elementary spin graph functions can alternatively be described as the simultaneous $I$-eigenspace for the one-dimensional
$I$-module $\bm{\chi}_{\bm{\pi}}\in I^\wedge$,
\[
\mathcal{H}^{\bm{\pi}}=\mathcal{H}_{\bm{\chi}_{\bm{\pi}}}.
\]

Hence condition (a) from \S\ref{2-6} always holds true for the quantum spin system,
\[
\mathcal{H}=\bigoplus_{\bm{\pi}\in (G^\wedge)^{E}}\mathcal{H}_{\bm{\chi}_{\bm{\pi}}},
\]
with $\mathcal{H}_{\bm{\chi}_{\bm{\pi}}}=\mathcal{H}^{\bm{\pi}}$ the finite dimensional space of $\bm{\pi}$-elementary spin graph functions.

\subsection{}\label{5-11}
  The following is the main result of the paper.
\begin{thm}
Let $\Gamma$ be a finite connected oriented graph, $G$ a connected compact Lie group, $\mathbf{K}=\prod_{v\in V}K_v$ with $K_v\subseteq G$ subgroups, and $\sigma: \mathbf{K}\rightarrow\textup{GL}(S)$ a finite dimensional representation.

The quantum spin system on $\mathcal{H}=\mathcal{H}_{\Gamma,G,S}$ as defined in \S\ref{5-9} is superintegrable if the following three conditions hold true:
\begin{enumerate}
\item[{\textup{(a)}}] $G$ is simply connected. 
\item[{\textup{(b)}}] For each $v\in V$, the local gauge group $K_v\subseteq G$ is closed and connected.
\item[{\textup{(c)}}] The representation $\sigma: \mathbf{K}\rightarrow\textup{GL}(S)$ is irreducible.
\end{enumerate}
\end{thm}
We give the proof of the theorem in \S\ref{6-10}.

\subsection{}\label{5-12}
Consider the quantum spin system with $\Gamma$ the oriented cycle graph with $n$ edges, $\mathbf{K}=G^V$ and $\sigma: \mathbf{K}\rightarrow\textup{GL}(V)$ a finite dimensional
representation, see \S\ref{4-17}.
By Theorem \ref{5-11} it is superintegrable when $G$ is simply connected and $\sigma$ is irreducible. The condition on $\sigma$ implies that $\sigma$ is equivalent to a
tensor product representation $\bm{\sigma}$ with its local representations $\sigma_v: G\rightarrow\textup{GL}(S_v)$ irreducible for all $v\in V$.

This quantum spin system can be made more explicit using the parametrisation of its moduli space $\mathcal{M}$ of graph $G$-connections in terms of a maximal torus $T\subset G$. The edge-component quadratic Hamiltonians $H_i$ then become explicit second-order $\textup{End}(S)$-valued differential operator on $T$ of Calogero-Moser type. The differences $H_i-H_{i-1}$ are first-order commuting differential operators called asymptotic Knizhnik-Zamolodchikov operators, which can be entirely described in terms of Felder's classical trigonometric dynamical $r$-matrix (see \cite{ES,St,Re3}). This provides the interpretation of this quantum spin system as a quantum periodic spin Calogero-Moser chain
\cite{RS}. 

For the special case $n=1$, the superintegrability of the quantum periodic Calogero-Moser spin system
was discussed in \cite{Re2}.  

\subsection{}\label{5-13}
Consider now the quantum spin system with $\Gamma$ the linearly ordered linear graph with $n$ edges, local gauge groups of the form
\eqref{localalg} with $H,K\subseteq G$ closed connected subgroups, and $\sigma: \mathbf{K}\rightarrow\textup{GL}(S)$ a finite dimensional representation of the associated gauge group $\mathbf{K}$ (see \S\ref{4-18}). By Theorem \ref{5-11} this quantum spin system is superintegrable when $G$ is simply connected and $\sigma$ is irreducible. The condition on $\sigma$ implies that $S\simeq L\otimes S_2\otimes\cdots\otimes S_n\otimes N$ with $L$ an irreducible $H$-representation, $N$ an irreducible $K$ representation and $S_j$ irreducible $G$-representations.

This quantum spin system can be made more concrete when $H=K$ is the connected component of the identity of a fix-point subgroup $G^\Theta$ of an involution $\Theta$ of $G$, using an appropriate parametrisation of its moduli space $\mathcal{M}$ of graph $G$-connections in terms of an appropriate subtorus $A\subset G$. The edge-component quadratic Hamiltonians $H_i$ then become second-order $\textup{End}(S)$-valued differential operator on $A$ of Calogero-Moser type and $H_i-H_{i-1}$ are asymptotic boundary Knizhnik-Zamolodchikov operators, which are first order differential operators involving folded classical dynamical $r$-matrices and associated dynamical $k$-matrices (see \cite{SR,RS,St,Re3}). This provides the interpretation of this quantum spin system as a quantum open spin Calogero-Moser chain \cite{RS}. 

\section{Conditions for superintegrability}\label{6}
In this section we provide a proof of the sufficient conditions ensuring superintegrability of the quantum spin systems defined in \S\ref{5} (see Theorem \ref{5-11}). 
We retain the notations and conventions of \S\ref{5}. In particular, $\Gamma$ is an oriented finite graph, $G$ is a connected compact Lie group, and 
$\mathbf{K}=\prod_{v\in V}K_v$ with subgroups $K_v\subseteq G$. 

We take as finite dimensional $\mathbf{K}$-representation of the quantum system a tensor product representation $\bm{\sigma}: \mathbf{K}\rightarrow\textup{GL}(\mathbf{S})$ (see \S\ref{4-10}). 
We furthermore fix an irreducible finite dimensional tensor product representation $\bm{\pi}: G^E\rightarrow\textup{GL}(\mathbf{M}^{\bm{\pi}})$, with local irreducible $G$-representations $\pi_e: G\rightarrow\textup{GL}(M^{\pi_e})$.

Finally, we write $\mathfrak{g}_0$ for the Lie algebra of $G$, and $\mathfrak{g}$ for its complexification.
\subsection{}\label{6-1}
For $v\in V$ consider the linear isomorphism
\begin{equation}\label{tauv}
\tau_v: \textup{Hom}(S_v^*,\mathbf{M}^{\pi_{\mathcal{S}(v)}})\overset{\sim}{\longrightarrow}\mathbf{M}^{\pi_{\mathcal{S}(v)}}
\otimes S_v
\end{equation}
defined by 
\[
\tau_v(\phi_v):=\sum_{t_v}\phi_v(u_{t_v}^{(v),*})\otimes u_{t_v}^{(v)}
\] 
where $\{u_t^{(v)}\}_{t}$ is a basis of $S_v$ and $\{u_t^{(v),*}\}_{t}$ is the corresponding dual basis of $S_v^*$.

It is often convenient to expand $\phi_v(u_{t_v}^{(v),*})$ in terms of the tensor product basis of $\mathbf{M}^{\pi_{\mathcal{S}(v)}}$ (see \S\ref{4-16}). Its expansion coefficients
will be denoted by $\phi_v\lbrack t_v;\bm{i},\bm{j}\rbrack\in\mathbb{C}$, 
\[
\phi_v(u_{t_v}^{(v),*})=\sum_{\bm{i}\in\mathcal{I}(v|s)}\sum_{\bm{j}\in\mathcal{I}(v|t)}\phi_v\lbrack t_v;\bm{i},\bm{j}\rbrack\,\bigl(\bm{m}_{\bm{i}}(v|s)^*\otimes
\bm{m}_{\bm{j}}(v|t)\bigr),
\]
so that
\begin{equation}\label{tauv}
\tau_v(\phi_v)=\sum_{\bm{i}\in\mathcal{I}(v|s)}\sum_{\bm{j}\in\mathcal{I}(v|t)}\sum_{t_v}\phi_v\lbrack t_v;\bm{i},\bm{j}\rbrack\,\bigl(\bm{m}_{\bm{i}}(v|s)^*\otimes
\bm{m}_{\bm{j}}(v|t)\bigr)\otimes u_{t_v}^{(v)}.
\end{equation}
\subsection{}\label{6-2}
Turn $\textup{Hom}(S_v^*,\mathbf{M}^{\pi_{\mathcal{S}(v)}})$ into a $K_v$-representation, with action
\[
k_v\cdot T:=\pi_{\mathcal{S}(v)}(k_v)\circ T\circ 
\sigma_v^*(k_v^{-1}),\qquad\quad k_v\in K_v.
\]
The linear map $\tau_v$ (see \eqref{tauv}) is $K_v$-linear, with the $K_v$-action on the codomain of $\tau_v$ 
as defined in \S\ref{4-14}. Hence $\tau_v$ restricts to a linear isomorphism
\begin{equation*}
\tau_v: \textup{Hom}_{K_v}(S_v^*,\mathbf{M}^{\pi_{\mathcal{S}(v)}})\overset{\sim}{\longrightarrow}\bigl(\mathbf{M}^{\pi_{\mathcal{S}(v)}}
\otimes S_v\bigr)^{K_v}.
\end{equation*}

\subsection{}\label{6-3}

Recall the isomorphism $\Psi^{\bm{\pi}}$ defined in \S\ref{4-15}. It follows from \S\ref{6-2} that
\[
\Upsilon^{\bm{\pi}}:=\Psi^{\bm{\pi}}\circ\Bigl(\bigotimes_{v\in V}\tau_v\Bigr):\,\,
\bigotimes_{v\in V}\textup{Hom}(S_v^*,\mathbf{M}^{\pi_{\mathcal{S}(v)}})\overset{\sim}{\longrightarrow}
\mathcal{R}^{\bm{\pi}}(G^E)\otimes\mathbf{S}
\]
is a $\mathbf{K}$-linear isomorphism,
where the domain of $\Upsilon^{\bm{\pi}}$ is viewed as $\mathbf{K}$-representation relative to 
the tensor product action of $\mathbf{K}=\prod_{v\in V}K_v$.
The map $\Upsilon^{\bm{\pi}}$ restricts to a linear isomorphism
\[
\Upsilon^{\bm{\pi}}: \bigotimes_{v\in V}\textup{Hom}_{K_v}(S_v^*,\mathbf{M}^{\pi_{\mathcal{S}(v)}})\overset{\sim}{\longrightarrow}
\mathcal{H}^{\bm{\pi}}.
\]

\subsection{}\label{6-4}

Let $\mathcal{I}$ be the set of  sequences $(i_e)_{e\in E}$ with $i_e\in\mathcal{I}_{e}$. Consider the tensor product basis $\{\mathbf{m}_{\bm{i}}\}_{\bm{i}\in\mathcal{I}}$ of $\mathbf{M}^{\bm{\pi}}$, where
\[
\mathbf{m}_{\bm{i}}:=\bigotimes_{e\in E}m_{i_e,e},
\] 
and write $\mathbf{m}_{\bm{i}}^*:=\bigotimes_{e\in E}m_{i_e,e}^*$ for the corresponding dual basis elements of $(\mathbf{M}^{\bm{\pi}})^*\simeq
\otimes_{e\in E}M^{\pi_e^*}$, cf. \S\ref{4-9}.
For $\bm{i}\in\mathcal{I}$ write 
\[
\bm{i}_{v|s}:=(i_e)_{e\in\mathcal{S}(v|s)}\in\mathcal{I}(v|s),\qquad\quad
\bm{i}_{v|t}:=(i_e)_{e\in\mathcal{S}(v|t)}\in\mathcal{I}(v|t).
\]
A direct computation using \eqref{tauv} then leads to the formula
\begin{equation}\label{lineariso}
\Upsilon^{\bm{\pi}}\Bigl(\bigotimes_{v\in V}\phi_v\Bigr)=
\sum_{\bm{i},\bm{j}\in\mathcal{I}}c_{{}\bm{m}_{\bm{i}}^*,\bm{m}_{\bm{j}}}^{\bm{\pi}}\otimes
\Bigl(\bigotimes_{v\in V}\bigl(\sum_{t_v}\phi_v\lbrack t_v; \bm{i}_{v|s},\bm{j}_{v|t}\rbrack\,u_{t_v}^{(v)}\bigr)\Bigr)
\end{equation}
for $\phi_v\in\textup{Hom}(S_v^*,\mathbf{M}^{\pi_{\mathcal{S}(v)}})$.

\subsection{}\label{6-5}
Consider the tensor product algebra
\[
U(\mathfrak{g})^{(v)}:=U(\mathfrak{g})^{\otimes\#\mathcal{S}(v|s)}\otimes U(\mathfrak{g})^{\otimes\#\mathcal{S}(v|t)}.
\]
A pure tensor in $U(\mathfrak{g})^{(v)}$
is denoted by $\mathbf{X}_{v|s}\otimes\mathbf{Y}_{v|t}$ with
\[
\mathbf{X}_{v|s}=\bigotimes_{e\in\mathcal{S}(v|s)}X_e,\qquad \mathbf{Y}_{v|t}=\bigotimes_{e^\prime\in\mathcal{S}(v|t)}Y_{e^\prime},
\] 
where we order the tensor products along the total orders on $\mathcal{S}(v|s)$ and $\mathcal{S}(v|t)$ induced by the total order on $E$.
In \S\ref{4-13} we considered the space 
\[
\mathbf{M}^{\pi_{\mathcal{S}(v)}}=\Bigl(\bigotimes_{e\in\mathcal{S}(v|s)}M^{\pi_e^*}\Bigr)\otimes\Bigl(\bigotimes_{e^\prime\in\mathcal{S}(v|t)}M^{\pi_{e^\prime}}\Bigr),
\]
as $K_v$-representation space relative to the diagonal $K_v$-action $\pi_{\mathcal{S}(v)}$. Differentiating the $G$-action turns $M^{\pi_e^*}$ and $M^{\pi_{e^\prime}}$
into irreducible $U(\mathfrak{g})$-modules, and hence $\mathbf{M}^{\pi_{\mathcal{S}(v)}}$ into an irreducible $U(\mathfrak{g})^{(v)}$-module via the diagonal 
$U(\mathfrak{g})^{(v)}$-action.

We view the linear space
\[
\textup{Hom}(S_v^*,\mathbf{M}^{\pi_{\mathcal{S}(v)}})
\]
as $U(\mathfrak{g})^{(v)}$-module, with $U(\mathfrak{g})^{(v)}$ acting on its co-domain,
\begin{equation}\label{codomaction}
\bigl((\mathbf{X}_{v|s}\otimes\mathbf{Y}_{v|t})\cdot T\bigr)(\xi):=(\mathbf{X}_{v|s}\otimes\mathbf{Y}_{v|t})\cdot (T(\xi))
\end{equation}
for $\mathbf{X}_{v|s}\otimes\mathbf{Y}_{v|t}\in U(\mathfrak{g})^{(v)}$, $T\in \textup{Hom}(S_v^*,\mathbf{M}^{\pi_{\mathcal{S}(v)}})$ and $\xi\in S_v^*$.

\subsection{}\label{6-6}
The local gauge group $K_v$ acts by algebra automorphisms on $U(\mathfrak{g})^{(v)}$ via the diagonal adjoint action,
\[
k_v\bullet_v\bigl(\mathbf{X}_{v|s}\otimes\mathbf{Y}_{v|t}\bigr):=
\Bigl(\bigotimes_{e\in\mathcal{S}(v|s)}\textup{Ad}(k_v)X_e\Bigr)\otimes\Bigl(\bigotimes_{e^\prime\in\mathcal{S}(v|t)}\textup{Ad}(k_v)Y_{e^\prime}\Bigr).
\]
We then have
\[
k_v\cdot\bigl((\mathbf{X}_{v|s}\otimes\mathbf{Y}_{v|t})\cdot B\bigr)=\bigl(k_v\bullet_v(\mathbf{X}_{v|s}\otimes\mathbf{Y}_{v|t})\bigr)\cdot\bigl(k_v\cdot B\bigr)
\]
for $k_v\in K_v$, $\mathbf{X}_{v|s}\otimes\mathbf{Y}_{v|t}\in U(\mathfrak{g})^{(v)}$ and
$B\in\mathbf{M}^{\pi_{\mathcal{S}(v)}}$. 

Let $(U(\mathfrak{g})^{(v)})^{K_v}$ be the algebra
of $K_v$-invariant elements in $U(\mathfrak{g})^{(v)}$ relative to the $K_v$-action $\bullet_v$.
It follows from \S\ref{6-2} and \S\ref{6-5} that the space 
\[
\textup{Hom}_{K_v}(S_v^*,\mathbf{M}^{\pi_{\mathcal{S}(v)}})
\]
of $K_v$-intertwiners is a $(U(\mathfrak{g})^{(v)})^{K_v}$-module, with the action 
on $\textup{Hom}_{K_v}(S_v^*,\mathbf{M}^{\pi_{\mathcal{S}(v)}})
$ given by \eqref{codomaction}.

\subsection{}\label{6-7}
Consider the algebra isomorphism
\begin{equation}\label{algiso}
U((\mathfrak{g}^E)^{\times 2})\overset{\sim}{\longrightarrow}\bigotimes_{v\in V}U(\mathfrak{g})^{(v)}
\end{equation}
defined by 
\[
\mathbf{X}\otimes\mathbf{Y}
\mapsto \bigotimes_{v\in V}\bigl(\mathbf{X}_{v|s}\otimes\mathbf{Y}_{v|t}\bigr)
\]
for $\mathbf{X}=\bigotimes_{e\in E}X_e$ and $\mathbf{Y}=\bigotimes_{e^\prime\in E}Y_{e^\prime}$, where
\[
\mathbf{X}_{v|s}=\bigotimes_{e\in\mathcal{S}(v|s)}X_e,\qquad\qquad
\mathbf{Y}_{v|t}=\bigotimes_{e^\prime\in\mathcal{S}(v|t)}Y_{e^\prime}.
\]

Consider the tensor product action of the gauge group $\mathbf{K}=\prod_{v\in V}K_v$ on the co-domain 
$\bigotimes_{v\in V}U(\mathfrak{g})^{(v)}$
of the algebra isomorphism \eqref{algiso},
with the $K_v$-action on
$U(\mathfrak{g})^{(v)}$ as defined in \S\ref{6-6}. A direct check shows that the algebra isomorphism \eqref{algiso} is $\mathbf{K}$-linear, with
$\mathbf{K}$ acting on $U((\mathfrak{g}^E)^{\times 2})$ according to Lemma \ref{5-5}. The algebra isomorphism \eqref{algiso} thus restricts to an algebra isomorphism
\begin{equation}\label{algisoK}
U((\mathfrak{g}^E)^{\times 2})^{\mathbf{K}}\overset{\sim}{\longrightarrow}\bigotimes_{v\in V}(U(\mathfrak{g})^{(v)})^{K_v}.
\end{equation}

\subsection{}\label{6-8}
Endow
\begin{equation}\label{vHom}
\bigotimes_{v\in V}\textup{Hom}(S_v^*,\mathbf{M}^{\pi_{\mathcal{S}(v)}})
\end{equation}
with the tensor product action of $\bigotimes_{v\in V}U(\mathfrak{g})^{(v)}$. We reinterpret this as  
an action of $U((\mathfrak{g}^E)^{\times 2})$ via the algebra isomorphism \eqref{algiso}.

By \S\ref{5-2} and \S\ref{5-7} the universal enveloping algebra $U((\mathfrak{g}^E)^{\times 2})$ also acts on the space $\mathcal{R}^{\bm{\pi}}(G^E)\otimes\mathbf{S}$ of $\mathbf{S}$-valued representative functions on $G^E$ by
\[
(\mathbf{X}\otimes\mathbf{Y})\cdot(f\otimes u):=D_{\mathbf{X}\otimes\mathbf{Y}}(f)\otimes u
\]
for $\mathbf{X}, \mathbf{Y}\in U(\mathfrak{g}^E)$, $f\in\mathcal{R}^{\bm{\pi}}(G^E)$ and $u\in\mathbf{S}$.
\begin{lem}
The $\mathbf{K}$-linear isomorphism
\[
\Upsilon^{\bm{\pi}}: \bigotimes_{v\in V}\textup{Hom}(S_v^*,\mathbf{M}^{\pi_{\mathcal{S}(v)}})\overset{\sim}{\longrightarrow}
\mathcal{R}^{\bm{\pi}}(G^E)\otimes\mathbf{S}
\]
as defined in \S\ref{6-3}, is $U((\mathfrak{g}^E)^{\times 2})$-linear.
\end{lem}
\begin{proof}
Using the notations from \S\ref{6-7} we have
\begin{equation*}
\begin{split}
\Upsilon^{\bm{\pi}}\Bigl((\mathbf{X}\otimes\mathbf{Y})\cdot\Bigl(\bigotimes_{v\in V}\phi_v\Bigr)\Bigr)&=\Upsilon^{\bm{\pi}}\Bigl(\bigotimes_{v\in V}(\mathbf{X}_{v|s}\otimes\mathbf{X}_{v|t})\cdot
\phi_v\Bigr)\\
&=\sum_{\bm{i},\bm{j}\in\mathcal{I}}c_{\mathbf{X}\cdot\mathbf{m}_{\bm{i}}^*,\mathbf{Y}\cdot\mathbf{m}_{\bm{j}}}^{\bm{\pi}}\otimes
\Bigl(\bigotimes_{v\in V}\bigl(\sum_{t_v}\phi_v\lbrack t_v;\bm{i}_{v|s},\bm{j}_{v|t}\rbrack\, u_{t_v}^{(v)}\bigr)\Bigr)
\end{split}
\end{equation*}
(here $\mathbf{X}\cdot\mathbf{m}_{\bm{i}}^*$ and $\mathbf{Y}\cdot\mathbf{m}_{\bm{j}}$ refer to the $U(\mathfrak{g}^E)$-action 
on $(\mathbf{M}^{\bm{\pi}})^*$ and $\mathbf{M}^{\bm{\pi}}$, obtained by differentiating the $G^E$-action).
The second equality follows from the expansion formula
\begin{equation*}
\begin{split}
&\bigl((\mathbf{X}_{v|s}\otimes\mathbf{Y}_{v|t})\cdot\phi_v\bigr)\lbrack t_v;\bm{i}_{v|s}^\prime,\bm{j}_{v|t}^\prime\rbrack=\\
&\,\,=\sum_{\bm{i}_{v|s},\bm{j}_{v|t}}\bigl(\mathbf{X}_{v|s}\cdot\mathbf{m}_{\bm{i}_{v|s}}(v|s)^*\bigr)
(\mathbf{m}_{\bm{i}_{v|s}^\prime}(v|s))\mathbf{m}_{\bm{j}_{v|t}^\prime}(v|t)^*(\mathbf{Y}_{v|t}\cdot\mathbf{m}_{\bm{j}_{v|t}}(v|t))\phi_v\lbrack t_v;\bm{i}_{v|s},\bm{j}_{v|t}\rbrack
\end{split}
\end{equation*}
with the sums in the left hand side taken over $\bm{i}_{v|s}\in\mathcal{I}(v|s)$ and $\bm{j}_{v|t}\in\mathcal{I}(v|t)$, and the fact that
\begin{equation*}
\begin{split}
\sum_{\bm{i}^\prime\in\mathcal{I}}\Bigl(\prod_{v\in V}\bigl(\mathbf{X}_{v|s}\cdot\mathbf{m}_{\bm{i}_{v|s}}(v|s)^*\bigr)(\mathbf{m}_{\bm{i}_{v|s}^\prime}(v|s))\Bigr)\mathbf{m}_{\bm{i}^\prime}^*&=\sum_{\bm{i}^\prime\in\mathcal{I}}(\mathbf{X}\cdot\mathbf{m}_{\bm{i}}^*)(\mathbf{m}_{\bm{i}^\prime})\mathbf{m}_{\bm{i}^\prime}^*=
\mathbf{X}\cdot\mathbf{m}_{\bm{i}}^*,\\
\sum_{\bm{j}^\prime\in\mathcal{I}}\Bigl(\prod_{v\in V}\mathbf{m}_{\bm{j}_{v|t}^\prime}(v|t)^*\bigl(\mathbf{Y}_{v|t}\cdot\mathbf{m}_{\bm{j}_{v|t}}(v|t)\bigr)\Bigr)\mathbf{m}_{\bm{j}^\prime}&=\sum_{\bm{j}^\prime\in\mathcal{I}}\mathbf{m}_{\bm{j}^\prime}^*(\mathbf{Y}\cdot\mathbf{m}_{\bm{j}})\mathbf{m}_{\bm{j}^\prime}=\mathbf{Y}\cdot \mathbf{m}_{\bm{j}}.
\end{split}
\end{equation*}
The result now follows from the fact that
\[
c_{\mathbf{X}\cdot\mathbf{m}_{\bm{i}}^*,\mathbf{Y}\cdot\mathbf{m}_{\bm{j}}}^{\bm{\pi}}=D_{\mathbf{X}\otimes\mathbf{Y}}(c_{\mathbf{m}_{\bm{i}}^*,\mathbf{m}_{\bm{j}}}^{\bm{\pi}}).
\]
\end{proof}
\subsection{}\label{6-9}
The results of \S\ref{6-6}-\S\ref{6-8} immediately lead to the following conclusion.
\begin{cor}
\hfill
\begin{enumerate}
\item[{\bf (a)}]
The isomorphism $\Upsilon^{\bm{\pi}}$ restricts to a $U((\mathfrak{g}^{E})^{\times 2})^{\mathbf{K}}$-linear isomorphism
\[
\Upsilon^{\bm{\pi}}:\,\, \bigotimes_{v\in V}\textup{Hom}_{K_v}\bigl(S_v^*,\mathbf{M}^{\pi_{\mathcal{S}(v)}}\bigr)\overset{\sim}{\longrightarrow}
\mathcal{H}^{\bm{\pi}}.
\]
\item[{\bf (b)}] $\mathcal{H}^{\bm{\pi}}$ is an irreducible $U((\mathfrak{g}^{E})^{\times 2})^{\mathbf{K}}$-module iff
$\textup{Hom}_{K_v}(S_v^*,\mathbf{M}^{\pi_{\mathcal{S}(v)}})$ is an irreducible $(U(\mathfrak{g})^{(v)})^{K_v}$-module
for all $v\in V$.
\end{enumerate}
\end{cor}

\subsection{}\label{6-10}
We have now all in the required ingredients for the proof of Theorem \ref{5-11}. 
\begin{proof}[Proof of Theorem \ref{5-11}]
Suppose that the three conditions (a)-(c) in Theorem \ref{5-11} hold true. In view of \S\ref{4-10}, we may assume 
without loss of generality that the $\mathbf{K}$-representation $\sigma$ is 
a tensor product representation $\bm{\sigma}$ with irreducible local representations $\sigma_v: K_v\rightarrow\textup{GL}(S_v)$. Let $\bm{\pi}\in (G^\wedge)^{E}$
such that $\mathcal{H}^{\bm{\pi}}\not=0$. We need to show that $\mathcal{H}^{\bm{\pi}}$ is an irreducible $U((\mathfrak{g}^{E})^{\times 2})^{\mathbf{K}}$-module.

Denote by $\mathfrak{k}_v$ the complexified Lie algebra of $K_v$. Set
\[
e(v):=\#\mathcal{S}(v|s)+\#\mathcal{S}(v|t)
\]
(which might be strictly larger than $\#\mathcal{S}(v)$ since $\Gamma$ may have loops).
Note that $e(v)>0$ since $\Gamma$ is connected. 

By Corollary \ref{6-9} and the fact that $K_v$ is connected, it suffices to show that 
$\textup{Hom}_{\mathfrak{k}_v}(S_v^*,\mathbf{M}^{\pi_{\mathcal{S}(v)}})$ is an irreducible $U(\mathfrak{g}^{\times e(v)})^{\mathfrak{k}_v^{(e(v))}}$-module, 
where $\mathfrak{k}_v^{(e(v))}\subseteq\mathfrak{g}^{\times e(v)}$ is the image $\mathfrak{k}_v$
under the diagonal embedding $\delta_{\mathfrak{g}}^{(e(v))}: \mathfrak{g}\hookrightarrow\mathfrak{g}^{\times e(v)}$, see \S\ref{3-9}.

Note that $\mathfrak{g}$ is semisimple since $G$ is simply connected, and $\mathfrak{k}_v$ is reductive in $\mathfrak{g}$ by \S\ref{3-5}.
Hence $\mathfrak{g}^{\times e(v)}$ is a reduction extension of $\mathfrak{k}_v^{(e(v))}$, see Proposition \ref{3-9}. Corollary \ref{3-17}
then implies that $\textup{Hom}_{\mathfrak{k}_v}(S_v^*,\mathbf{M}^{\pi_{\mathcal{S}(v)}})$ is an irreducible $U(\mathfrak{g}^{\times e(v)})^{\mathfrak{k}_v^{(e(v))}}$-module.
\end{proof}



\begin{thebibliography}{99}
\bibitem{AGS} A. Yu. Alekseev, H. Grosse, V. Schomerus, {\it Combinatorial quantization of the Hamiltonian Chern-Simons Theory}, Comm. Math. Phys. {\bf 172} (1995), 317--358. 
\bibitem{AMR} J.E. Andersen, J. Mattes, N. Reshetikhin, {\it The Poisson structure on the moduli space of flat connections and chord diagrams},
Topology {\bf 15} (1996), 1069--1083.
\bibitem{AR} S. Arthamonov, N. Reshetikhin, {\it Superintegrable systems on moduli spaces of flat connections},  Comm. Math. Phys. {\bf 386}
(2021), 1337--1381.
\bibitem{BR1} E. Buffenoir, Ph. Roche, {\it Two dimensional lattice gauge theory based on a quantum group}, Comm. Math. Phys. {\bf 170} (1995),
669--698. 
\bibitem{BR2} E. Buffenoir, Ph. Roche, {\it  Link invariants and combinatorial quantization of Hamiltonian Chern Simons theory}, Comm. Math. Phys.
{\bf 181} (1996), 331--365.
\bibitem{Di} J. Dixmier, {\it Enveloping algebras},  Graduate studes in Math., {\bf 11}. Amer. Math. Soc., Providence, RI, 1996.
\bibitem{Eetal} P. Etingof, O. Golberg, S. Hensel, T. Liu, A. Schwendner, D. Vaintrob, E. Yudovina, {\it Introduction to Representation Theory}, Student Math. Library {\bf 59}, Amer. Math. Soc. (2011).
\bibitem{ES} P. Etingof, O. Schiffmann, {\it Twisted traces of intertwiners for Kac-Moody algebras and classical dynamical $r$-matrices corresponding to 
Belavin-Drinfeld triples}, Math. Res. Lett. {\bf 6} (1999), 593--612.
\bibitem{ESV} P. Etingof, O. Schiffmann, A. Varchenko, {\it Traces of intertwiners for quantum groups and difference equations}, Lett. Math. Phys. {\bf 62} (2002), 143--158. 
\bibitem{EV} P. Etingof, A. Varchenko, {\it Traces of intertwiners for quantum groups and difference equations. I}, Duke Math. J {\bf 104} (2000), 391--432.
\bibitem{FR} V.V. Fock, A.A. Rosly, {\it Poisson structure on moduli of flat connections on Riemann surfaces and the $r$-matrix}, Moscow Seminar in Math. Phys.,
67--86, Amer. Math. Soc. Transl. Ser. 2, {\bf 191}, Adv. Math. Sci., {\bf 43}, Amer. Math. Soc., Providence, RI, 1999.
\bibitem{HC0} Harish-Chandra, {\it Spherical functions on a  semisimple Lie group I}, Amer. J. Math. {\bf 80} (1958), 241--310.
\bibitem{HC} Harish-Chandra, {\it Representations of semisimple Lie groups. II}, Trans. Amer. Math. Soc. {\bf 76} (1954), 26--65.
\bibitem{LM} J. Leposwky, G.W. McCollum, {\it On the determination of irreducible modules by restriction to a subalgebra}, 
Trans. Amer. Math. Soc. {\bf 176} (1973), 45--57.
\bibitem{Ne} Ju. A. Neretin, {\it Differential operators on Lie groups} (in russian), Uspekhi Mat. Nauk {\bf 36}, no. 2(218) (1981), 195--196.
\bibitem{Op} E.M. Opdam, {\it Root systems and hypergeometric functions IV}, Compositio Math. {\bf 67} (1988), 191--209.
\bibitem{Re1} N. Reshetikhin, {\it Degenerate integrability of the Spin Calogero-Moser Systems and the Duality with the Spin Ruijsenaars Systems}, Lett. Math. Phys. {\bf 63} (2003), 55--71.
\bibitem{Re2} N. Reshetikhin, {\it Degenerate integrability of quantum spin Calogero-Moser systems}, Lett. Math. Phys. {\bf 107} (2017), 187--200.
\bibitem{Re23} N. Reshetikhin, {\it Periodic and open classical spin Calogero-Moser chains}, arXiv:2302.14281.
\bibitem{Re3} N. Reshetikhin, {\it Spin Calogero-Moser periodic chains and two dimensional Yang-Mills theory with corners}, arXiv:2303.10579.
\bibitem{RS} N. Reshetikhin, J.V. Stokman, {\it Asymptotic boundary KZB operators and quantum Calogero-Moser spin chains}, in ``Hypergeometry, integrability and Lie theory'',
205--241, Contemp. Math. {\bf 780}, Amer. Math. Soc., Providence, RI, 2022.
\bibitem{RSz} P. Roche, A. Szenes, {\it Functionals on noncommutative deformations of moduli spaces of flat connections}, Adv. Math. {\bf 168} (2002), 133--192.
\bibitem{St} J.V. Stokman, {\it Folded and contracted solutions of coupled classical dynamical Yang-Baxter and reflection equations}, Indag. Math. (N.S.) {\bf 32} (2021),
1372--1411.
\bibitem{SR} J.V. Stokman, N. Reshetikhin, 
{\it $N$-point spherical functions and asymptotic boundary KZB equations}, Invent. Math. {\bf 229} (2022), 1--86.
\bibitem{Wa} N. Wallach, {\it Real reductive groups I.} Pure and Applied Mathematics, {\bf 132}. Academic Press, Inc., Boston, MA, 1988.
\end{thebibliography}
\end{document}